\documentclass[letterpaper,11pt]{article}
\usepackage[margin=1in]{geometry}  

\usepackage{listings}
\usepackage{amsmath}
\usepackage{amsthm}
\usepackage{tikz}
\usepackage{caption,subcaption}
\usepackage{array}
\usepackage{mdwmath}
\usepackage{multirow}
\usepackage{mdwtab}
\usepackage{eqparbox}
\usepackage{multicol}
\usepackage{amsfonts}
\usepackage{tikz}
\usepackage{multirow,bigstrut,threeparttable}
\usepackage{amsthm}
\usepackage{array}
\usepackage{bbm}
\usepackage{epstopdf}
\usepackage{mdwmath}
\usepackage{mdwtab}
\usepackage{eqparbox}
\usepackage{tikz}
\usepackage{latexsym}
\usepackage{cite}
\usepackage{amssymb}
\usepackage{bm}
\usepackage{amssymb}
\usepackage{graphicx}
\usepackage{mathrsfs}
\usepackage{epsfig}
\usepackage{psfrag}
\usepackage{setspace}
\usepackage[
            CJKbookmarks=true,
            bookmarksnumbered=true,
            bookmarksopen=true,
            colorlinks=true,
            citecolor=blue, 
            linkcolor=blue,
            anchorcolor=blue, 
            urlcolor=blue
            ]{hyperref}
\usepackage{algorithm}
\usepackage{algpseudocode}
\usepackage{stfloats}

\input xy
\xyoption{all}

\theoremstyle{plain}
\newtheorem{theorem}{Theorem}
\newtheorem{lemma}{Lemma}

\newtheorem{corollary}{Corollary}

\theoremstyle{definition}

\def\NN{\mathbb{N}}

\def\RR{\mathbb{R}}

\def\calA{\mathcal{A}}

\def\calC{\mathcal{C}}

\def\calM{\mathcal{M}}
\def\calN{\mathcal{N}}

\def\calQ{\mathcal{Q}}

\def\calX{\mathcal{X}}

\def\bE{\mathbf{E}}

\def\bP{\mathbf{P}}

\def\bR{\mathbf{R}}

\def\1{\mathbbm{1}}
\def\var{\mathsf{Var}}

\usepackage{booktabs}
\usepackage{adjustbox}
\usepackage{multirow}
\usepackage{listings}

\theoremstyle{plain}

\newtheorem{definition}{Definition}

\def \bP {\mathbb{P}}
\def \bE {\mathbb{E}}
\def \bR {\mathbb{R}}

\def \var {\mathsf{Var}}

\usepackage{xspace}

\newcommand{\stepa}[1]{\overset{\rm (a)}{#1}}
\newcommand{\stepb}[1]{\overset{\rm (b)}{#1}}
\newcommand{\stepc}[1]{\overset{\rm (c)}{#1}}
\newcommand{\stepd}[1]{\overset{\rm (d)}{#1}}
\newcommand{\stepe}[1]{\overset{\rm (e)}{#1}}

\newcommand{\naturals}{\mathbb{N}}

\newcommand{\Poi}{\mathsf{Poi}}

\definecolor{myblue}{rgb}{.8, .8, 1}
\definecolor{mathblue}{rgb}{0.2472, 0.24, 0.6} 
\definecolor{mathred}{rgb}{0.6, 0.24, 0.442893}
\definecolor{mathyellow}{rgb}{0.6, 0.547014, 0.24}

\begin{document}

\title{On the Competitive Analysis and High Accuracy Optimality of Profile Maximum Likelihood}
\author{Yanjun Han and Kirankumar Shiragur\thanks{Yanjun Han is with the Department of Electrical Engineering, Stanford University, email: \url{yjhan@stanford.edu}. Kirankumar Shiragur is with the Department of Management Science \& Engineering, Stanford University, email: \url{shiragur@stanford.edu}. Kirankumar Shiragur was supported by Stanford Data Science Scholarship.} }

\maketitle

\begin{abstract}
A striking result of Acharya et al. \cite{acharya2017unified} showed that to estimate symmetric properties of discrete distributions, plugging in the distribution that maximizes the likelihood of observed multiset of frequencies, also known as the profile maximum likelihood (PML) distribution, is competitive compared with any estimators regardless of the symmetric property. Specifically, given $n$ observations from the discrete distribution, if some estimator incurs an error $\varepsilon$ with probability at most $\delta$, then plugging in the PML distribution incurs an error $2\varepsilon$ with probability at most $\delta\cdot \exp(3\sqrt{n})$. In this paper, we strengthen the above result and show that using a careful chaining argument, the error probability can be reduced to $\delta^{1-c}\cdot \exp(c'n^{1/3+c})$ for arbitrarily small constants $c>0$ and some constant $c'>0$. The improved competitive analysis leads to the optimality of the PML plug-in approach for estimating various symmetric properties within higher accuracy $\varepsilon\gg n^{-1/3}$. In particular, we show that the PML distribution is an optimal estimator of the sorted distribution: it is $\varepsilon$-close in sorted $\ell_1$ distance to the true distribution with support size $k$ for any $n=\Omega(k/(\varepsilon^2 \log k))$ and $\varepsilon \gg n^{-1/3}$, which are the information-theoretically optimal sample complexity and the largest error regime where the classical empirical distribution is sub-optimal, respectively. 

In order to strengthen the analysis of the PML, a key ingredient is to employ novel ``continuity" properties of the PML distributions and construct a chain of suitable quantized PMLs, or ``coverings". We also construct a novel approximation-based estimator for the sorted distribution with a near-optimal concentration property without any sample splitting, where as a byproduct we obtain better trade-offs between the polynomial approximation error and the maximum magnitude of coefficients in the Poisson approximation of $1$-Lipschitz functions. 
\end{abstract}
\newpage
\tableofcontents
\newpage
\section{Introduction}
Symmetric property estimation is a fundamental problem in computer science and statistics with  applications in neuroscience \cite{RWDB99}, physics \cite{VBBVP12}, ecology \cite{Chao84, Chao92, BF93, CCGLMCL12}, and beyond \cite{PW96, ASNRMAS01}. Over the past decade the optimal sample complexity for estimating various symmetric properties has been studied and resolved, including entropy \cite{Paninski2003,Paninski2004,Valiant--Valiant2011,Jiao--Venkat--Han--Weissman2015minimax,wu2016minimax}, support size \cite{Valiant--Valiant2013estimating, wu2019chebyshev}, support coverage \cite{orlitsky2016optimal,ZVVKCSLSDM16}, distance to uniformity \cite{Valiant--Valiant2011power,jiao2018minimax}, sorted $\ell_{1}$ distance \cite{Valiant--Valiant2011power, han2018local}, R\'enyi entropy~\cite{AOST14, AOST17}, and many others. A natural question is to design a universal estimator, an estimator that is sample competitive with respect to all symmetric properties.

Acharya et al. \cite{acharya2017unified} showed that computing the profile maximum likelihood (PML), i.e. the distribution that maximizes the likelihood of observed multiset of frequencies, and returning the property value of this distribution attains optimal rates of sample complexity for estimating various symmetric properties. Formally, Acharya et al. showed that given $n$ i.i.d. samples from an unknown distribution $p$, if there exists an estimator for estimating a symmetric property $f$ achieving error $\varepsilon$ with success probability $1-\delta$, then this PML-based plug-in approach achieves error $2\varepsilon$ with probability $1-\delta e^{3\sqrt{n}}$. This result further implied an unified approach for estimating symmetric properties such as support size, support coverage, entropy, and distance to uniformity when the estimation error $\varepsilon \gg n^{-1/4}$, as the error probability $\delta$ for estimating these properties behaves as a Gaussian tail $\exp(-\Omega(n\varepsilon^2))$. The authors in \cite{acharya2017unified} further argued that a $\beta$-approximate PML\footnote{A $\beta$-approximate PML is the distribution that achieves a multiplicative $\beta$-approximation to the PML objective.} for $\beta > \exp(-n^{1-c})$ and any constant $c>0$ suffices to get a unified approach that is sample optimal for $\varepsilon \gg n^{- \min(1/4,c/2)}$.

Acharya et al. \cite{acharya2017unified} posed two important open questions. In the first question they asked for a polynomial time algorithm to compute an $\exp(-n^{1-c})$-approximate PML distribution for some constant $c>0$. This open question was resolved recently in \cite{charikar2019efficient}, where the authors present the first polynomial time algorithm to compute an $\exp(-n^{2/3} \log n)$-approximate PML distribution; the approximation factor was later improved to $\exp(-\sqrt{n} \log n)$ by \cite{ACSS20}. In their second open question, Acharya et al. suggested that the requirement $\varepsilon \gg n^{-1/4}$, or the amplified failure probability $\delta e^{3\sqrt{n}}$, may just be an artifact of the analysis and asked for a better analysis of PML. Moreover in most of the previous works, the PML distribution is mainly studied in the context of symmetric property estimation and its fundamental understanding as a \emph{probability distribution} remains largely unknown with the notable exceptions \cite{hao2019broad,hao2020profile}. 

In this work we address the second open question. We show that the previous penalty term $\exp(3\sqrt{n})$, which upper bounds the cardinality of all possible PML distributions, is pessimistic and does not account for the internal relationships between different PML distributions. Instead, we propose a novel continuity property of the PML distribution and show that pooling together similar PML distributions helps significantly reduce the effective cardinality term. Specifically, by a careful chaining argument, the error probability of the PML-based plug-in approach can be reduced from $\delta\exp(3\sqrt{n})$ to $\delta^{1-c}\exp(c'n^{1/3+c})$ for arbitrarily small constant $c>0$ with a large constant $c'>0$\footnote{As $\delta$ is typically at most $\exp(-n^{t})$ for some absolute constant $t>0$, the exponent of the quantity $\delta^{1-c}$ has the same order of magnitude as that of $\delta$, so we do not lose much from $\delta$ to $\delta^{1-c}$.}. Consequently, for all the aforementioned symmetric properties, the PML-based approach attains the optimal sample complexity for any estimation error $\varepsilon \gg n^{-1/3}$. 

The improvement from $n^{-1/4}$ to $n^{-1/3}$, albeit seemingly tiny, in fact has striking consequences. First, it helps us establish a complete analogy between the PML distribution and the classical maximum likelihood (ML) distribution (i.e. the empirical distribution), where just like the ML distribution is an optimal estimator of the hidden distribution, the PML distribution is an optimal estimator of the \emph{sorted} hidden distribution. Specifically, based on $n$ i.i.d. observations from the distribution $p$ with support size $k$, the ML distribution $\widehat{p}^{\text{ML}}$ satisfies,
\begin{align*}
\bP\left(\|\widehat{p}^{\text{ML}} - p\|_1 \ge c\sqrt{\frac{k}{n}} + \varepsilon \right) \le \exp(-c'n\varepsilon^2)
\end{align*}
for some constants $c,c'>0$ and any $\varepsilon>0$, which is rate-optimal in estimating the distribution $p$ under the $\ell_1$ (or equivalently, the total variation) distance \cite{han2015minimax,kamath2015learning}. Similarly, we show that the PML distribution $\widehat{p}^{\text{PML}}$ satisfies
\begin{align}\label{eq:PML_sorted}
\bP\left(\|\widehat{p}^{\text{PML}} - p\|_{1, \text{sorted}} \ge c\sqrt{\frac{k}{n\log n}} + \varepsilon \right) \le \exp(-c'n^{1-\delta}\varepsilon^2)
\end{align}
for $\varepsilon\gg n^{-1/3}$ and any small constant $\delta>0$, which is also known to be rate-optimal among all possible estimators in estimating the distribution $p$ under the sorted $\ell_1$ distance based on $n$ samples\cite{han2018local}. Compared with the unsorted distribution and ML, \eqref{eq:PML_sorted} shows that estimating the sorted distribution is strictly easier with a logarithmic improvement on the sample complexity. Moreover, after reaching the sample complexity, \eqref{eq:PML_sorted} also gives a near sub-Gaussian concentration property of the PML for sorted distribution estimation, just like ML for unsorted distribution estimation. We remark that the additional assumption $\varepsilon\gg n^{-1/3}$ in \eqref{eq:PML_sorted} is necessary as well: it is known in \cite{han2018local} that the ML distribution becomes optimal in estimating the sorted distribution when $\varepsilon \le n^{-1/3}$, and a phase transition on the minimax rate occurs at $\varepsilon\asymp n^{-1/3}$. Hence, the condition $\varepsilon\gg n^{-1/3}$ is the information-theoretically optimal range one may possibly hope for an improvement over the ML distribution. The above comparison shows that estimating the unsorted distribution is fundamentally \emph{different} from the sorted distribution, and the PML distribution is an optimal estimator for the latter task. We note that a similar analogy (i.e. the same form of \eqref{eq:PML_sorted}) was also shown in previous work \cite{hao2019broad,hao2020profile} when $\varepsilon\gg n^{-1/8}$, which applied the competitive result of \cite{acharya2017unified} and an indirect argument factoring through the optimal estimation of general $1$-Lipschitz functionals. In contrast, our work improves upon the fundamental PML analysis in \cite{acharya2017unified} and directly proposes a new estimator for the sorted distribution with low sensitivity. We provide more comparisons in the related work. 

Second, we establish an analogy between the PML-based plug-in approach and the local moment matching (LMM) approach proposed in \cite{han2018local}, both of which are \emph{unified} approaches in the sense that they simply plug in some distribution independent of the specific symmetric property in hand; the former approach simply plugs the PML distribution into the target property, and the latter plugs in the LMM distribution. Albeit constructed based on different principles, both approaches are minimax rate-optimal in estimating the sorted distribution, and they both attain the optimal sample complexity for various symmetric properties above the common error threshold $\varepsilon\gg n^{-1/3}$. We conjecture that this common threshold may be the price one need to pay for any unified methodology of symmetric property estimation without knowing the specific property beforehand, which will also imply that the exponent of the amplified error probability $\delta^{1-c}\exp(c'n^{1/3+c})$ of PML cannot be improved to $O(n^{1/3-c})$ in general and thus the tightness of our general PML analysis.\footnote{This conjecture was recently confirmed to be true in the follow-up work \cite{han2020high}.}

To apply the competitive result, a key intermediate step in establishing the optimality of the PML distribution in estimating the sorted distribution is to find an estimator which along with other properties also satisfies an exponentially small failure probability as in \eqref{eq:PML_sorted}. However, the only known minimax rate-optimal estimator for the sorted distribution, i.e. the LMM estimator in \cite{han2018local}, has an inverse polynomial failure probability like $O(n^{-5})$ which is not sufficient for our purposes, as the sample splitting technique and solving a linear program in each local interval makes it unstable to perturbations. In our work, we propose a new estimator of the sorted distribution that has optimal sample complexity and is the first estimator with sorted $\ell_1$ loss enjoying near-optimal concentration properties (therefore an exponentially small failure probability). The crux of our estimator construction is to abandon sample splitting, a technique to acquire independence commonly used in almost all past works in estimating symmetric properties, and to construct a global and explicit linear program, both in order to increase the stability of the resulting estimator to slight perturbations. In particular, the removal of sample splitting enables us to obtain better trade-offs between the polynomial approximation error and the maximum magnitude of coefficients in the Poisson approximation of $1$-Lipschitz functions, a side result in approximation theory which also improves the key technical lemma in \cite{vinayak2019maximum,vinayak2019optimal}.

In summary, the main contributions of this paper are as follows: 
\begin{enumerate}
	\item We provide a first improvement of the competitive analysis of PML in \cite{acharya2017unified}, where we reduce the amplification factor on the error probability from $\exp(O(\sqrt{n}))$ to $\exp(O(n^{1/3+c}))$ for any constant $c>0$. Consequently, we establish the high accuracy optimality of the PML plug-in approach for estimating various symmetric properties within any accuracy parameter $\varepsilon\gg n^{-1/3}$. Our new competitive analysis relies on a novel continuity property of the PML distributions as well as a careful and involved chaining argument, both of which could be of independent technical interest. 
	\item We provide a novel estimator for the sorted discrete distribution with rate-optimal sample complexity and near-optimal Gaussian tails. Improving on the local moment matching idea in \cite{han2018local}, the construction of our estimator crucially relies on a new surrogate loss and a notion of smoothed local moments to remove the usual sample splitting technique and increase the stability. This idea also leads to a better trade-off between the polynomial approximation error and the maximum magnitude of coefficients in the Poisson approximation of $1$-Lipschitz functions, a side result in approximation theory. 
	\item Combining the above two results, we show that the PML distribution is optimal in estimating the sorted distribution for the largest accuracy range $\varepsilon\gg n^{-1/3}$, and therefore establish a complete analogy between the PML and the classical ML distributions. 
\end{enumerate}

\subsection{Our Techniques}
\noindent\textbf{Refined PML Analysis.} Let $\Phi_n$ be the set of all possible profiles generated by $n$ observations, and for each profile $\phi\in \Phi_n$, let $p_\phi$ be the PML distribution associated with the profile $\phi$. The target is to show that, for a random profile $\phi$ generated from the data distribution $p$, the associated PML distribution $p_\phi$ will be close to $p$ in some proper distance. A big challenge in the analysis of PML is that the PML distribution $p_\phi$ is the solution to a complicated and non-convex optimization program, so very few things can be said about $p_\phi$ except for its defining property $\bP(p_\phi, \phi)\ge \bP(p,\phi)$ for all possible distributions $p$, where $\bP(p,\phi)$ denotes the probability of observing the profile $\phi$ if the true data distribution is $p$. To overcome this difficulty, \cite{acharya2017unified} makes a key observation which we interpret in a slightly different way below: for each data distribution $p$, there is an associated subset of ``good profiles" $G\subseteq \Phi_n$ such that $\bP(p,G) \triangleq \sum_{\phi\in G}\bP(p, \phi) \ge 1- \delta$, and $p_\phi$ is $2\varepsilon$-close to $p$ whenever $\bP(p_\phi, G)> \delta$ (cf. Lemma \ref{lemma:PML}). Consequently, the probability that the random PML distribution $p_\phi$ with $\phi\sim p$ is at least $2\varepsilon$-far from $p$ is at most
\begin{align}\label{eq:error_prob}
\sum_{\phi\in \Phi_n} \bP(p,\phi)\1(\bP(p_\phi,G)\le \delta) \le \sum_{\phi\in G} \bP(p,\phi)\1(\bP(p_\phi,G)\le \delta) + \delta, 
\end{align}
as $\bP(p,\Phi_n\backslash G)\le \delta$. Therefore, to upper bound the above probability, it remains to show that for each profile $\phi \in G$, the PML distribution $p_\phi$ puts a large probability mass on $G$. To this end, \cite{acharya2017unified} uses the defining property of PML, i.e. $\bP(p_\phi, G) \ge \bP(p_\phi, \phi) \ge \bP(p,\phi)$. Consequently, the above error probability \eqref{eq:error_prob} is further upper bounded by
\begin{align*}
\sum_{\phi\in G} \bP(p,\phi)\1(\bP(p,\phi)\le \delta) + \delta \le (|\Phi_n|+1)\delta, 
\end{align*}
where the cardinality of $\Phi_n$ satisfies $|\Phi_n|\le\exp(3\sqrt{n})$ (cf. Lemma \ref{lemma:profile}), giving the claimed result in \cite{acharya2017unified} (plus an additive small $\delta$). 

How should we improve the analysis above? A possibly loose step is the na\"{i}ve inequality $\bP(p_\phi, G) \ge \bP(p_\phi, \phi)$, which is tight only if the PML distribution $p_\phi$ puts most of its mass on the single profile $\{\phi\}$, or equivalently, $\bP(p_\phi, \phi')$ is very small for all other profiles $\phi'\neq \phi$. In contrast, if we could show that $\bP(p_\phi, \phi) \approx \bP(p_{\phi'}, \phi)$ for any $\phi'\neq \phi$, i.e. other PML distributions $p_{\phi'}$ also have a large probability on $\phi$ comparable to that of the maximizer $p_\phi$, then $\bP(p_{\phi'}, G)\le \delta$ would imply that 
\begin{align*}
\delta \ge \sum_{\phi\in G} \bP(p_{\phi'}, \phi) \approx  \sum_{\phi\in G} \bP(p_{\phi}, \phi) \ge \sum_{\phi\in G} \bP(p,\phi). 
\end{align*}
Consequently, the error probability in \eqref{eq:error_prob} would be as small as $O(\delta)$! However, there are two key challenges in the above ideal arguments. First, it is unlikely to conclude that $\bP(p_\phi, \phi) \approx \bP(p_{\phi'}, \phi)$ for \emph{all} pairs of profiles $\phi \neq \phi'$, especially when these profiles are ``far apart". Second, and more importantly, how can we prove that $\bP(p_\phi, \phi) \approx \bP(p_{\phi'}, \phi)$ even for a \emph{given} pair $(\phi,\phi')$ if the only property we know about PML distributions is the defining property $\bP(p_\phi, \phi) \ge \bP(p,\phi)$ for all $p$? Note that if we \emph{only} knew the defining property of $p_\phi$, it would be definitely possible that different $p_\phi$'s are mutually singular and the argument in \cite{acharya2017unified} is tight. 

To overcome the above difficulties, our key argument is to utilize a novel ``continuity" property of PML distributions. Specifically, to address the first challenge, we show that $\bP(p_\phi, \phi) \approx \bP(p_{\phi'}, \phi)$ holds \emph{locally} when the profiles $\phi$ and $\phi'$ belong to the same local ball; in other words, the PML distribution $p_\phi$ changes in a continuous way from one profile $\phi$ to another. For the second challenge, i.e. how to show the above relationship even locally, we observe the fact that profiles are deterministic functions of \emph{histograms}, and utilize the data processing inequality to show a \emph{covering} property of all distributions of the form $\bP(p,\cdot)$ on profiles. Mathematically, we show that there exists three functions $r(n), s(n), t(n)$ such that there is a class $\calN$ of distributions such that $|\calN| \le r(n)$, and any distribution $p$ gives rise to a distribution $q\in \calN$ such that for any $S\subseteq \Phi_n$, the following approximation properties hold: 
\begin{align}
\bP(p,S) \ge \bP(q,S)^{1/(1-s(n))}\cdot \exp(-t(n)), \label{eq:approximation_1} \\
\bP(q,S) \ge \bP(p,S)^{1/(1-s(n))}\cdot \exp(-t(n)). \label{eq:approximation_2}
\end{align}
Ideally, we want $r(n)$ and $t(n)$ to be small positive integers, and $s(n)\in (0,1)$ to be a small number close to zero; however, there exist tradeoffs among these terms, and we defer the precise statements to Lemma \ref{lemma:continuity}. The above result implies that the class $\calN$ of distributions immediately results in a \emph{covering} of all PML distributions: if $p_{\phi} $ and $p_{\phi'}$ correspond to the same element of $\calN$ (i.e. in the same cover), then repeated applications of \eqref{eq:approximation_1} and \eqref{eq:approximation_2} show that for any $S\subseteq \Phi_n$, the two probabilities $\bP(p_\phi,S)$ and $\bP(p_{\phi'},S)$ are close in a proper multiplicative sense. In other words, each cover consists of multiple similar PML distributions, a manifestation of the continuity property of PML distributions. Therefore, the above continuity property simultaneously addresses both challenges, providing the precise definitions of locality and the approximation $\bP(p_\phi, \phi) \approx \bP(p_{\phi'}, \phi)$. 

Now the help from the above continuity property to an upper bound of the error probability in \eqref{eq:error_prob} is a careful modification of the analysis in the previous ideal scenario. Specifically, for each covering of profiles $G_q \subseteq G$ corresponding to the element $q\in\calN$, some repeated applications of \eqref{eq:approximation_1} and \eqref{eq:approximation_2} indicate that the error probability summing over all local profiles $\phi\in G_q$ is at most $\delta^{1-o(1)}\exp(O(s(n)|\Phi_n|+t(n)))$; the details are deferred to Section \ref{subsec:warmup}. As $|\calN|\le r(n)$, the total error probability in \eqref{eq:error_prob} is at most $\delta^{1-o(1)}r(n)\exp(O(s(n)|\Phi_n|+t(n)))$, where the best exponent after trading off $r(n),s(n),t(n)$ is $O(n^{3/8}\log n)$, already a notable improvement over the exponent $O(\sqrt{n})$ in \cite{acharya2017unified}. For the further improvement from $O(n^{3/8}\log n)$ to $O(n^{1/3+c})$ for any constant $c>0$, instead of working on a single layer of coverings $\calN$, we construct a careful chain of coverings $\calN_1 \supseteq \calN_2 \supseteq \cdots$ (as well as different functions $r_i(n), s_i(n), t_i(n)$ at the $i$-th layer) and make use of different continuity properties at each layer. We defer the details to Section \ref{subsec:chaining}. 

\vspace{1.5mm}
\noindent\textbf{Sorted Distribution Estimation.} The key difficulty in constructing an estimator of the sorted distribution satisfying \eqref{eq:PML_sorted}, in addition to its minimax rate-optimality, is to make the estimator stable to slight perturbations in observed samples. This task is relatively simple when the target property is a \emph{scalar}, where a typical estimator, as a function of the i.i.d. observations $(X_1,\cdots,X_n)$ or equivalently the histograms $(h_1,\cdots,h_k)$, is a sum of base estimators taking the form $f_i(h_i), i\in [k]$ with thus easily analyzable perturbations. Nearly all past works \cite{acharya2017unified,hao2019broad,hao2019unified,hao2020profile} fall into the above category to establish sub-Gaussian tails for proper univariate functions $f_i$. However, for sorted distribution estimation, the target property is a \emph{high-dimensional vector}, and all past known estimators such as the ones in \cite{Valiant--Valiant2013estimating,han2018local} are highly complex and cannot be written in the above simple form. For example, the complete minimax rate-optimal estimator in \cite{han2018local} involves a sample splitting technique to assign each domain element to one of many local intervals, whereas in each local interval a linear program is solved to form the final estimator. Hence, there are two reasons for the instability of the resulting estimator. First, solving a linear program can be treated as an \emph{inverse problem}, which might be ill-conditioned without any closed-form solution, so it is challenging to translate the closeness in the LP parameters to the closeness of solutions. Second, the sample splitting technique makes it easy to have some domain element assigned to another local interval after perturbation, which in turn results in a substantial change on the final estimator. We remark that nearly all past works on property estimation (not limited to establishing sub-Gaussian tails) relied on sample splitting to acquire independence; we refer to the related work section for details. 

To overcome the difficulty in analyzing the inverse problem, in this paper we propose a new objective function which serves as a \emph{surrogate loss} for the sorted $\ell_1$ distance. Consequently, the minimization of this surrogate loss will also imply the closeness in the target sorted $\ell_1$ distance. Specifically, we compute some statistic $T=T(X_1,\cdots,X_n)\in \bR^m$ from the samples, and define the estimator $\widehat{p}$ to be the minimizer of some loss $L: \calM\times \bR^m \to \bR_+$, i.e. $\widehat{p} = \arg\min_p L(p,T)$. The function $L$ is chosen to be a surrogate loss in the sense that, for any discrete distributions $p,q$ and any (random) statistic $T$, it holds deterministically that
\begin{align}\label{eq:surrogate}
\|p - q\|_{1,\text{sorted}} \le L(p,T) + L(q,T) + \text{fixed error terms}. 
\end{align}
In view of \eqref{eq:surrogate} and the definition that $L(\widehat{p},T) \le L(p,T)$ for the true data distribution $p$, we have $\|p - \widehat{p}\|_{1,\text{sorted}}\le 2L(p,T) + \text{fixed error terms}$. Consequently, we mitigate the drawback caused by the inverse problem and shift from handling potentially very complicated $\widehat{p}$ to handling a simpler quantity $L(p,T)$ where $p$ is the true data distribution. To convince the readers that $L$ is indeed a simpler quantity, we remark that our surrogate loss takes the form $L(p,T) = \sum_{j=1}^m w_j |f_j(p) - T_j|$, which is simply a weighted sum of the estimation errors of the statistic $T_j$ in estimating a properly chosen property $f_j(p)$ of the data distribution $p$. Consequently, it should be expected to become much easier to control the expectation and the perturbation property of $L(p,T)$. 

The sample splitting comes into play in the functions $f_j(p)$. In fact, the local moment matching idea in \cite{han2018local} suggests that ideally we should choose $f_j(p)$ to take the form $\sum_{i=1}^k p_i^d\1(p_i\in I)$, which is some moment of the probability vector $p=(p_1,\cdots,p_k)$ restricted to a given local interval $I\subseteq [0,1]$. However, this quantity cannot be reliably estimated by any statistic $T_j$ based on samples, as $\1(p_i\in I)$ is hard to locate when $p_i$ is close to the boundary of $I$. To mitigate this drawback, in the past sample splitting is used: split each empirical frequency $\widehat{p}_i$ into independent copies $\widehat{p}_{i,1}, \widehat{p}_{i,2}$, then the modified quantity $\sum_{i=1}^k p_i^d\1(\widehat{p}_{i,1}\in I)$ admits an \emph{unbiased} estimator based on $\widehat{p}_{i,2}$'s. Although this new quantity works for obtaining a minimax rate-optimal estimator, it still introduces instability when $\widehat{p}_{i,1}$ changes from one local interval to another. In this paper, we introduce the concept of \emph{smoothed local moments}, where the property $f_j(p)$ is now chosen to be $\sum_{i=1}^k p_i^d\cdot \bP(\mathsf{Poi}(np_i/2)\in nI/2)$ and the Poisson probability can be viewed as a soft version of the hard decision $\1(p_i\in I)$. This new quantity satisfies both advantages: first, it admits an unbiased estimator and can be estimated reliably; second, the unbiased estimator for this quantity changes smoothly in perturbations. This idea removes the sample splitting in general property estimation problems, and is the key to the desired high probability results. 

\vspace{1.5mm}
\noindent\textbf{Implications on Approximation Theory.} The above idea to remove sample splitting is useful not only in obtaining high probability results in estimation theory, but also in providing a general tool to piece together different approaches in different local intervals smoothly. We illustrate this point using an example in approximation theory. Let $f$ be any $1$-Lipschitz function on $[0,1]$ with $f(0)=0$, then what is the best trade-off between the Bernstein polynomial approximation error $\max_{x\in [0,1]} |f(x) - \sum_{i=0}^n a_iB_{n,i}(x)|$, where $B_{n,i}(x) = \binom{n}{i}x^i(1-x)^{n-i}$ is the Bernstein basis, and the maximum coefficient $\max_i |a_i|$? This question, although purely approximation theoretic, plays a central role in analyzing the population maximum likelihood estimator under the empirical Bayes framework \cite{vinayak2019maximum,vinayak2019optimal}, where it was shown that one may achieve an approximation error $O(1/\sqrt{n\log n})$ while $\max_i |a_i| = O(n^{1.5+\varepsilon})$. However, we show that it is in fact possible to achieve $\max_i |a_i| = O(1)$ with the same approximation error, which in turn enlarges the applicable range of the results in \cite{vinayak2019maximum,vinayak2019optimal}. The crucial idea in establishing this new result is to perform polynomial approximation locally, and then utilize our previous smoothing idea to piece together different local polynomials. 

\subsection{Related Work}
Following the principle of maximum likelihood proposed by Ronald Fisher and its success in the classical H\'{a}jek--Le Cam asymptotic theory \cite[Chapter 9]{Vandervaart2000}, \cite{orlitsky2004modeling} suggested the concept of the \emph{profile maximum likelihood} with the symbol labels discarded, and \cite{orlitsky2011estimating} summarized some basic properties of the PML. Since then, a series of work were devoted to the computational side of the PML, where approaches such as the Bethe/Sinkhorn approximation \cite{vontobel2012bethe,vontobel2014bethe}, the EM algorithm \cite{orlitsky2004algorithms}, algebraic approaches \cite{acharya2010exact} and a dynamic programming approach \cite{pavlichin2019approximate} were proposed to compute an approximate PML. Very recently, \cite{charikar2019efficient} proposed the first polynomial-time algorithm that provably computes an $\exp(-n^{2/3} \log n)$-approximate PML distribution. This result was later improved in \cite{ACSS20}, where the authors present a polynomial time algorithm to compute an $\exp(-\sqrt{n} \log n)$-approximate PML distribution.

Recently, \cite{acharya2017unified} studied the statistical side of the PML, where they showed that PML based plug-in approach is competitive to the optimal estimator in symmetric functional estimation, with error probability $\delta$ amplified to $\delta\exp(3\sqrt{n})$. Consequently, to analyze the performance of PML one may simply turn to another problem, i.e. analyze the performance of the optimal estimator, then invoke the above competitive result. This line of thought was adopted in \cite{hao2019broad}, which studied the optimal estimation of various symmetric properties. There were also other attempts to improve the PML analysis. For example, the work \cite{charikar2019general,hao2019broad} independently proposed modifications of PML called pseudo/truncated PML with better statistical properties for a \emph{subset} of symmetric properties, while did not touch the vanilla PML analysis in general. Using a property specific approach, the authors in \cite{charikar2019general} showed that the PML based estimator is minimax optimal in estimating the support size for all regimes of error parameter $\varepsilon$, where the support of PML can be computed explicitly and the analysis does not extend to any other properties. Under the empirical Bayes setting, the notion of the population MLE in \cite{vinayak2019maximum,vinayak2019optimal} is similar in spirit to the PML, while their population likelihood admits a much simpler additive form and is therefore easier to deal with. Finally, the most recent work \cite{hao2020profile} showed that the exponent $3\sqrt{n}$ for the vanilla PML in \cite{acharya2017unified} could be improved to a distribution-specific quantity called \emph{profile entropy}, but this quantity is still $\Omega(\sqrt{n})$ for the worst-case distribution. In contrast, we give a generic improvement to $O(n^{1/3+c})$ even for worst-case distributions. Therefore, our work is the first improvement over \cite{acharya2017unified} on general statistical properties (or the competitive analysis) of the vanilla PML without any modifications. 

We also review the literature on symmetric property estimation, or the functional estimation in general, which has been extensively studied in the past decade. Specifically, the minimax-optimal estimation procedures have been found for many non-smooth 1D properties, including the entropy \cite{Paninski2003,Paninski2004,Valiant--Valiant2011,Jiao--Venkat--Han--Weissman2015minimax,wu2016minimax}, $L_1$ norm of the mean \cite{Cai--Low2011}, support size \cite{Valiant--Valiant2013estimating, wu2019chebyshev}, support coverage \cite{orlitsky2016optimal,ZVVKCSLSDM16}, distance to uniformity \cite{Valiant--Valiant2011power,jiao2018minimax}, and general 1-Lipschitz functions \cite{hao2019broad,hao2019unified}. Similar procedures can also be generalized to 2D functionals \cite{han2016minimax,bu2018estimation,jiao2018minimax} and nonparametric problems \cite{Lepski--Nemirovski--Spokoiny1999estimation,han2017estimation,han2017optimal}. We refer to the survey paper \cite{verdu2019empirical} for an overview of the results. However, two technical tricks were consistently used in these works. First, most constructed estimators were only proved to work in Poissonized models (where there is independence across different domain elements), with the only exception \cite{han2017optimal} which applied a careful Efron--Stein--Steele inequality to the i.i.d. sampling model. Second, most works relied on the sample splitting to carry out the two-stage estimation procedure summarized in \cite{Jiao--Venkat--Han--Weissman2015minimax} independently, with the notable exceptions \cite{orlitsky2016optimal,wu2019chebyshev} using explicit construction of linear estimators to get rid of the sample splitting. However, a general methodology of interpolating between different estimation regimes is still underexplored, and no concentration properties were proved in the above works. These issues prevent from constructing minimax optimal estimators in the original sampling model with good concentration properties, e.g. in the example of sorted distribution estimation. 

The general problem of sorted distribution estimation, or the estimation of the parameter multiset, was studied in \cite{Valiant--Valiant2013estimating,han2018local} for discrete distributions and \cite{rigollet2019uncoupled} for the Gaussian location model. The minimax optimal procedures in the above works are mainly based on method of moments dating back to \cite{pearson1894contributions}, which was also used in \cite{hardt2015tight,wu2018optimal} for learning Gaussian mixtures and \cite{kong2017spectrum,tian2017learning} for learning a population of parameters. However, the above performance guarantees only hold in expectation, and it is highly challenging to construct an estimator with an exponential concentration property due to the complicated and potential unstable nature of the method of moments. As for the property of PML in estimating sorted distributions, the pioneering work \cite{orlitsky2005convergence} established the consistency of PML, and \cite{orlitsky2011estimating,anevski2017estimating} proved an $O(n^{-1/2})$ convergence rate of the expected $\ell_\infty$ distance. For the sorted $\ell_1$ distance, the paper \cite{acharya2012estimating} (see also the thesis \cite{das2012competitive}) used the competitive analysis to show the optimality of the PML for $\varepsilon=\Omega(1/\log n)$ based on a modified estimator to \cite{Valiant--Valiant2013estimating}. The accuracy parameter $\varepsilon$ was later improved to $\varepsilon \ge n^{-0.1}$ and $\varepsilon \gg n^{-1/8}$ in recent work \cite{hao2019broad,hao2020profile}, which still does not reach the optimal threshold $\varepsilon\gg n^{-1/3}$ established in \cite{han2018local}. The main drawback is that the arguments in \cite{hao2019broad,hao2020profile} factor through the estimation of general $1$-Lipschitz functionals and are thus indirect, which seem insufficient to achieve the expected threshold $\varepsilon\gg n^{-1/4}$ even if one applies the known competitive result of \cite{acharya2017unified}. In contrast, we simultaneously improve the essential competitive analysis and provide a new estimator for the sorted distribution to feed into the competitive analysis, which in combination helps attain the optimal threshold $\varepsilon\gg n^{-1/3}$. Therefore, we provide a complete answer on the estimation performance of the PML of the sorted distribution in the sense that the PML is minimax rate-optimal in the largest possible error regime $\varepsilon \gg n^{-1/3}$, and there exists an explicit estimator which estimates the sorted distribution in an optimal way and with exponentially high probability.

\subsection{Organization}
The rest of the paper is organized as follows. Section \ref{sec:prelim} lists useful notations and preliminaries for this paper, and Section \ref{sec:main_result} summarizes the main results. Section \ref{sec.PML} presents the proof of Theorem \ref{thm:PML_main}, where we propose a notion of continuity on the PML distributions and apply a novel chaining argument. Section \ref{sec.estimator} presents the detailed construction of the optimal estimator for the sorted distribution and the roadmap of the proof of Theorem \ref{thm:sorted_main}. In particular, in Section \ref{subsec.poisson_approx} we obtain an improved trade-off between the polynomial approximation error and the maximum magnitude of coefficients in the Poisson approximation of $1$-Lipschitz functions, which may be of independent interest in approximation theory. Remaining proofs are relegated to the appendices. 


\newcommand{\R}{\mathbb{R}}
\newcommand{\nprob}[2]{\mathbb{P}\left(#1,#2\right)}
\newcommand{\pml}{p^{\text{\rm PML}}}
\newcommand{\kkdel}[1]{{\color{blue} Delete This:  #1}}
\newcommand{\constA}{A}
\newcommand{\estf}{\widehat{f}}
\newcommand{\csto}{c_{1}}
\newcommand{\cstt}{c_{2}}
\newcommand{\cstth}{c_{3}}
\newcommand{\setim}{S_{i,m}}

\section{Preliminaries}\label{sec:prelim}
\subsection{Notations}
Throughout the paper we adopt the following notations. Let $\NN$ be the set of all positive integers, and for $n\in \NN$, let $[n]\triangleq \{1,2,\ldots,n\}$. For reals $a,b\in\RR$, let $[a,b]$ and $(a,b)$ denote the closed and open intervals from $a$ to $b$, respectively. For a finite set $A$, let $|A|$ be the cardinality of $A$. Let $\calM$ denote a generic subset of all discrete distributions supported on $\NN$, where the specific choice of $\calM$ depends on the context\footnote{Typically $\calM$ is chosen to be the set of all discrete distributions supported on some finite domain $[k]$ (e.g. in the estimation of entropy, distance to uniformity, and sorted distribution), but there are a few exceptions. For example, in the support size estimation $\calM$ is the set of all discrete distributions supported on $\NN$ with all non-zero probability masses at least $1/k$, and in the support coverage estimation $\calM$ can be the entire set of discrete distributions on $\NN$.}, and for $k\in\NN$, let $\calM_k$ be the set of all discrete distributions supported on $[k]$. Let $\Poi(\lambda)$ be the Poisson distribution with rate parameter $\lambda\ge 0$, and $\mathsf{B}(n,p)$ be the Binomial distribution with number of trials $n$ and success probability $p\in [0,1]$. We adopt the following asymptotic notations. For non-negative sequences $\{a_n\}$ and $\{b_n\}$, we write $a_n\lesssim b_n$ (or $a_n = O(b_n)$) to denote that $\limsup_{n\to\infty}a_n/b_n<\infty$, and $a_n \gtrsim b_n$ (or $a_n = \Omega(b_n)$) to denote $b_n\lesssim a_n$, and $a_n\asymp b_n$ (or $a_n = \Theta(b_n)$) to denote both $a_n\lesssim b_n$ and $b_n\lesssim a_n$. Moreover, let $\widetilde{O}(\cdot)$, $ \widetilde{\Omega}(\cdot)$ and $ \widetilde{\Theta}(\cdot)$ denote the respective meanings within multiplicative poly-logarithmic factors in $n$. We also write $a_n\ll b_n$ to denote that $\limsup_{n\to\infty} n^{\varepsilon} a_n/b_n = 0$ for some small $\varepsilon>0$, and $a_n \gg b_n$ to denote $b_n\ll a_n$. 

\subsection{Preliminaries}
Throughout the paper we will consider the following i.i.d. sampling model: let $p\in\calM$ be an unknown discrete distribution, and the learner observes a sequence $x^n=(x_1,\ldots,x_n)$ of length $n$ drawn i.i.d. from the hidden distribution $p$. Specifically, the statistical model is given by
$$\nprob{p}{x^n} = \prod_{i \in [n]} p(x_i),$$
where $p(x)$ denotes the probability mass of the domain element $x\in \NN$. Based on the observation sequence $x^n$, we next review the definitions of \emph{histograms}\footnote{Histograms are also known as \emph{types} in the information theory literature.} and \emph{profiles}. 
\begin{definition}[Histogram]\label{def:histogram}
	Given a sequence $x^n \in [k]^n$ of length $n$, the \emph{histogram} of $x^n$ is defined as a vector $h = (h_1,\ldots,h_k)$ with $h_j$ being the number of occurrences of symbol $j \in [k]$, i.e.,
	\begin{align*}
	h_j = \sum_{i \in [n]} \1(x_i = j). 
	\end{align*}
\end{definition}

\begin{definition}[Profile]\label{def:profile}
	Given a sequence $x^n\in [k]^n$ of length $n$ with a histogram $h=(h_1,\cdots,h_k)$, the \emph{profile} $\Phi(x^n)$ of $x^n$ is defined as the vector $\phi = (\phi_1,\cdots,\phi_n)$ with $\phi_i$ being the number of symbols occurring $i$ times, i.e.,
	\begin{align*}
	\phi_i = \sum_{j\in [k]} \1(h_j = i). 
	\end{align*}
	Note that we exclude $\phi_0$ from the profile as we do not observe the unseen symbols and the domain size $k$ is unknown. Let $\Phi_n$ denote the set of all profiles generated by observations of length $n$. 
\end{definition}
The following lemma bounds the cardinality of distinct profiles. 
\begin{lemma}[\!\!\cite{hardy1918asymptotic}]\label{lemma:profile}
	For any $n\in \NN$, we have $|\Phi_n| \le \exp(3\sqrt{n})$. 
\end{lemma}
For a profile $\phi \in \Phi_n$ and a distribution $p \in \calM$, let 
\begin{align*}
\nprob{p}{\phi} = \sum_{x^n \in \NN^n: \Phi(x^n) = \phi} \nprob{p}{x^n}
\end{align*}
be the probability of observing a sample $x^n$ with profile $\phi$ under $p$. 
Further, for any subset $S \subseteq \Phi_n$, we use the $\nprob{p}{S}$ to denote the following: 
$$\nprob{p}{S} = \sum_{\phi \in S} \nprob{p}{\phi}.$$

The profile of a sequence is a sufficient statistic for any symmetric property estimation and it further motivates the following definition of profile maximum likelihood.
\begin{definition}[PML]
	Given a profile $\phi \in \Phi_n$, the \emph{profile maximum likelihood (PML)} distribution is defined as
	\begin{align*}
	\pml = \arg\max_{p\in \calM} \nprob{p}{\phi},
	\end{align*}
	where $\calM$ is the feasible set of all probability distributions when solving the PML distribution. Further for any $\beta \in (0,1]$, the distribution $p^{\mathrm{PML},\beta}\in\calM$ is a $\beta$-approximate PML distribution if it satisfies 
	\begin{align*}
	\nprob{p^{\mathrm{PML},\beta}}{\phi} \geq \beta \cdot \max_{p\in \calM} \nprob{p}{\phi}.
	\end{align*}
\end{definition} 

We call function $d: \calM\times \calM \to \bR_+$ a pseudo-metric if it is non-negative, symmetric and satisfies the triangle inequality. The pseudo-metric $d(p,q)$ may be measured directly based on norms of distributions (such as the $\ell_1$ and $\ell_2$ norms of $p-q$), or via properties of distributions, e.g. $d(p,q) = |f(p) - f(q)|$ is the distance between $p$ and $q$ based on a general property $f:\calM \to \bR$. In the above property estimation example, we may also interpret $d(p,q)$ as the estimation error of $f(p)$ when one uses a plug-in estimator $f(q)$ based on the distribution $q$. Let $\calA$ be a general action space containing all possible outputs of estimators. We call a loss function $L: \calA\times \calM \to \bR_+$ to be \emph{compatible with $d$} if and only if for any $p,q\in \calM$ and action $a\in \calA$, the following triangle inequality holds: 
\begin{align}\label{eq:compatible_loss}
d(p,q) \le L(a,p) + L(a,q). 
\end{align} 
In the above property estimation example, the action space is $\calA = \bR$ since each action corresponds to an estimate of $f(p)$. Hence, a natural compatible loss function is $L(a,p) = |a - f(p)|$, i.e. the difference between the estimate $a$ and the true property value $f(p)$. 

The main pseudo-metric $d$ considered in this work is the sorted $\ell_1$ distance between discrete distributions. Specifically, For discrete distributions $p,q\in \calM_k$, the sorted $\ell_1$ distance between $p$ and $q$ is defined as
\begin{align}\label{eq:sorted_l1}
\| p - q \|_1^< \triangleq \sum_{i=1}^k |p_{(i)} - q_{(i)}|, 
\end{align}
where $p_{(1)}\le p_{(2)} \le \cdots \le p_{(k)}$ and $q_{(1)}\le q_{(2)}\le \cdots\le q_{(k)}$ are the sorted versions of distributions $p$ and $q$, respectively. By a simple rearrangement inequality, the sorted $\ell_1$ distance is also the minimum $\ell_1$ distance between the unlabeled distributions $p$ and $q$ under all possible labelings (or permutations). There is also an equivalent formulation of the sorted $\ell_1$ distance in terms of the Wasserstein distance: for $p=(p_1,\cdots,p_k)\in \calM_k$, let the probability measure associated with $p$ be
\begin{align}\label{eq:measure_p}
\mu_p \triangleq \frac{1}{k}\sum_{i=1}^k \delta_{p_i},
\end{align}
where $\delta_x$ denotes the Dirac point mass at $x$. Then \cite[Lemma 7]{han2018local} shows that
\begin{align}\label{eq:sorted_Wasserstein}
\| p - q\|_1^< = k\cdot \text{W}_1(\mu_p, \mu_q),
\end{align}
where $\text{W}_1(\cdot,\cdot)$ is the Wasserstein-$1$ distance associated with the usual metric $d(x,y)=|x-y|$ on $\bR$ defined as
\begin{align*}
\text{W}_1(P,Q) \triangleq \inf_{X\sim P, Y \sim Q} \bE|X-Y|,
\end{align*}
and the infimum is taken over all possible couplings of $(P,Q)$. Hence, let the action space $\calA$ be the set of all possible probability measures, then the loss $L(\mu,p) = k\text{W}_1(\mu,\mu_p)$ is compatible with the sorted $\ell_1$ metric. Note that we slightly enlarged the action space and do not restrict $\mu$ to take the form of $\mu_q$ for some $q\in \calM_k$. 

Finally, we introduce a definition of $(\alpha,\beta)$-closeness between discrete distributions which will be useful in stating the main results. 
\begin{definition}[$(\alpha,\beta)$-closeness]\label{def:closeness}
	For any $\alpha,\beta\in (0,1)$ and a discrete distribution $p=(p_1,\cdots,p_k)$, we call that another distribution $p'=(p_1',\cdots,p_k')$ is $(\alpha,\beta)$-close to $p$ iff the following three conditions hold: 
	\begin{align*}
	\bullet ~p_i = 0  \Longrightarrow p_i' = 0 \quad \bullet ~p_i \le \alpha  \Longrightarrow p_i' \le \alpha \quad \bullet ~p_i > \alpha  \Longrightarrow \frac{p_i}{1+\beta} \le p_i' \le p_i.
	\end{align*}
\end{definition}
\section{Main Results}\label{sec:main_result}
In this section we state the main results of our paper. Our first result presents a generic competitive nature of the PML distribution compared with any other estimator depending only on the profile, with an improved competitive ratio on the probability of error. 

\begin{theorem}\label{thm:PML_main}
	Let $d: \calM\times \calM \to \bR_+$ be a pseudo-metric and $L$ be a loss function compatible with $d$. Suppose there exists an estimator $\widehat{T}$ depending only on the profile of $n$ i.i.d. observations from $p$ that satisfies
	\begin{align}\label{eq:PML_assumption}
	\bP(L(\widehat{T}, p) \ge \varepsilon) \le \delta, \qquad \forall p\in \calM. 
	\end{align}
	Then for any constants $A>0, c>0$, there exists an absolute constant $c'=c'(A,c)>0$ such that the PML distribution $\pml$ restricted to $\calM$ satisfies
	\begin{align*}
	\bP(d(\pml, p) \ge 2\varepsilon + \varepsilon' )\le \delta^{1-c}\cdot \exp(c'n^{1/3+c})
	\end{align*}
	for all $p\in \calM$, where $\varepsilon'=\sup\{d(p,p'): p, p'\in \calM, p'\text{ is }(n^{-A}, 3n^{-A/2})\text{ close to }p\}$. 
\end{theorem}


Informally, Theorem \ref{thm:PML_main} shows that the PML distribution is competitive in the sense that, if there exists an estimator $\widehat{T}$ depending only on the profile which is $\varepsilon$-close to the true distribution under a general loss with probability at least $1-\delta$, then the PML distribution is also $(2+o(1))\varepsilon$-close to the true distribution with probability at least $1-\delta^{1-c}\exp(c'n^{1/3+c})$. Hence, for estimators with the error probability $\delta$ smaller than $\exp(-n^{1/3+\delta'})$ for any $\delta'>0$, the PML distribution has the same error exponent as those estimators. 

To the authors' knowledge, Theorem \ref{thm:PML_main} is the first improvement on the error probability of the vanilla PML distribution over the only known result in \cite{acharya2017unified} that uses a simple union bound approach. Compared with \cite{acharya2017unified}, Theorem \ref{thm:PML_main} provides a refined error probability of the PML estimator where the multiplicative factor reduces from $\exp(3\sqrt{n})$ to $\exp(c'n^{1/3+c})$ for constant $c\to 0$. The dependence on $\delta$ becomes slightly worse (as it changes from $\delta$ to $\delta^{1-c}$), but it does not change the asymptotic order of the exponent as $\delta$ needs to be exponentially small for a non-trivial final probability of Theorem \ref{thm:PML_main}. The additional term $\varepsilon'$ is mostly technical, and for typical metric $d$ this quantity satisfies $\varepsilon'\le n^{-C}$ for an arbitrarily large exponent $C>0$ after choosing the parameter $A>0$ large enough (see the examples in Corollary \ref{cor:functional}). Since for most statistical estimation problems the error $\varepsilon$ of interest is at least the parametric rate $\Omega(n^{-1/2})$, the additional term $\varepsilon'$ will be much smaller than $\varepsilon$. 
Finally, by looking at the proof of Theorem \ref{thm:PML_main}, similar results also hold for approximate PML distributions: for any $\beta$-approximate PML with $\beta\in (0,1)$, under condition \eqref{eq:PML_assumption} it holds that
\begin{align*}
\bP(d(p^{\text{\rm PML}, \beta}, p) \ge 2\varepsilon + \varepsilon' )\le \frac{\delta^{1-c}}{\beta}\cdot \exp(c'n^{1/3+c})
\end{align*}
for all $p\in \calM_k$. Therefore, any $\exp(-O(n^{1/3}))$-approximate PML distribution gives the same rate as Theorem \ref{thm:PML_main}. 

Choosing $d(p,q) = |f(p) - f(q)|$ for general symmetric properties $f$, a direct consequence of Theorem \ref{thm:PML_main} is on the estimation performance of the PML plug-in approach in symmetric property estimation. Specifically, the following corollary shows that, for various symmetric property estimation problems, the PML plug-in estimator achieves the optimal sample complexity in a wider error regime than those established in previous works. 
\begin{corollary}\label{cor:functional}
	Let $n$ be the optimal sample complexity for estimating entropy, support size, support coverage, or the distance to uniformity within accuracy $\varepsilon$. If $\varepsilon \gg n^{-1/3}$, then the PML plug-in estimator attains the rate-optimal sample complexity for estimating these properties to accuracy $(2+o(1))\varepsilon$.
\end{corollary}

Similar improvements over the results of \cite{hao2019broad} in its applications can also be shown directly, which are omitted for brevity as different applications involve different assumptions. 

Our second result is specialized to some intrinsic properties of the PML distribution without introducing any external properties.
Specifically, we show that the PML distribution is an optimal estimator of the true distribution under the sorted $\ell_1$ distance. To apply the result of Theorem \ref{thm:PML_main}, we first need the following intermediate result on the estimation of sorted distribution. 
\begin{theorem}\label{thm:sorted_main}
	For any $k\ge 1$, $p\in \calM_k$ and $\delta>0$, there exist constants $c,c'>0$ and an estimator $\widehat{\mu}$ depending only on the profile of $n$ i.i.d observations from $p$ such that, if $\varepsilon\ge n^{\delta-1/3}$, then
	\begin{align*}
	\bP\left( k\text{\rm W}_1(\widehat{\mu}, \mu_p) \ge c\sqrt{\frac{k}{n\log n}} + \varepsilon \right) \le \exp(-c'n^{1-\delta}\varepsilon^2). 
	\end{align*}
\end{theorem}

The following corollary on the sample complexity is immediate from Theorem \ref{thm:sorted_main}. 
\begin{corollary}\label{cor:sample_complexity}
	For any $k\ge 1$ and $\delta>0$, there exist constants $c,c'>0$ and an estimator $\widehat{\mu}$ depending only on the profile of $n\ge k/(\varepsilon^2\log k)$ i.i.d observations from $p$ such that, if $\varepsilon\ge n^{\delta-1/3}$ and $p\in\calM_k$, then
	$$\bP\left(k\cdot\text{\rm W}_1(\widehat{\mu},\mu_p)\ge c\varepsilon \right) \leq \exp(-c'n^{1-\delta}\varepsilon^2)~.$$ 
\end{corollary}

We remark that by feeding Theorem \ref{thm:sorted_main} into the previously known competitive analysis in \cite{acharya2017unified}, we immediately achieve the optimality of the PML distribution in estimating the sorted distribution for accuracy parameters $\varepsilon\gg n^{-1/4}$, already a notable improvement over the work \cite{hao2019broad,hao2020profile} which required $\varepsilon\gg n^{-1/8}$.

Let us compare Theorem \ref{thm:sorted_main} with the performance of the empirical distribution $\widehat{p}$: the expected $\ell_1$ distance result in \cite{han2015minimax,kamath2015learning} together with the McDiarmid's inequality gives
\begin{align*}
\bP\left(\|\widehat{p} - p\|_1^< \ge c\sqrt{\frac{k}{n}} + \varepsilon \right) \le \exp(-c'n\varepsilon^2)
\end{align*}
for any $\varepsilon>0$. Hence, the result in Theorem \ref{thm:sorted_main} gives a logarithmic improvement over the empirical estimator on the expected sorted $\ell_1$ loss, also with a near-optimal exponential concentration. Moreover, it was proved in \cite{han2018local} that the minimax risk is $\Theta(\sqrt{k/(n\log n)}) + \widetilde{\Theta}(\sqrt{k/n}\wedge n^{-1/3})$, and therefore the optimal sample complexity of the sorted distribution estimation within error $\varepsilon$ is $\Theta(k/(\varepsilon^2\log k))$ if $\varepsilon\gg n^{-1/3}$, and is $\Theta(k/\varepsilon^2)$ otherwise. In other words, the error assumption $\varepsilon\gg n^{-1/3}$ exactly corresponds to the interesting regime where logarithmic improvements over the empirical estimator are possible, and the estimator in Theorem \ref{thm:sorted_main} achieves the optimal minimax rate in this regime.  

The previous work \cite{han2018local} obtained a similar result to Theorem \ref{thm:sorted_main} and matching minimax lower bounds, but a failure probability only polynomial in $n$, i.e. $O(n^{-5})$, was proved. Moreover, the estimator in \cite{han2018local} requires sample splitting and solves multiple feasibility programs, which is difficult to prove a super-polynomial failure probability and also computationally expensive to solve. In this paper, we construct the estimator $\widehat{\mu}$ by solving a \emph{single} linear program \emph{without} any sample splitting, with a super-polynomially small failure probability from which the applications of Theorem \ref{thm:PML_main} benefit. These new insights also generalize to the estimation of general symmetric properties, and even lead to an improved result in approximation theory. We refer to the full paper for more details. 

We are now ready to prove that the PML distribution is close to the sorted hidden distribution. Recall the choice of the pseudo-metric $d(p,q) = \|p-q\|_1^<$ and compatible loss $L(\mu,p) = k\text{W}_1(\mu,\mu_p)$ in the previous section. The next result is a direct consequence of Theorems \ref{thm:PML_main} and \ref{thm:sorted_main}. 

\begin{theorem}\label{thm:PML_sorted}
	For any $k\ge 1$, $p\in \calM_k$ and $\delta>0$, there exist constants $c,c'>0$ such that given $n$ i.i.d observations from distribution $p$, if $\varepsilon\ge n^{\delta-1/3}$, then
	\begin{align*}
	\bP\left( \|\pml - p\|_1^< \ge c\sqrt{\frac{k}{n\log n}} + \varepsilon \right) \le \exp(-c'n^{1-\delta}\varepsilon^2). 
	\end{align*}
\end{theorem}


Since $\varepsilon\ge n^{\delta-1/3}$, the failure probability $\exp(-c'n^{1-\delta}\varepsilon^2)$ of the optimal estimator $\widehat{\mu}$ in Theorem \ref{thm:sorted_main} is small enough to offset the amplification factor $\exp(c'n^{1/3+c})$ for the PML distribution in Theorem \ref{thm:PML_main}. The following corollary on the sample complexity of PML distribution in estimating sorted $\ell_1$ distance is immediate from Theorem \ref{thm:PML_sorted}. 

\begin{corollary}\label{cor:PML_sample_complexity}
	For any $k\ge 1$, $p\in \calM_k$ and $\delta>0$, there exist constants $c,c'>0$ such that given $n\ge k/(\varepsilon^2\log k)$ i.i.d observations from distribution $p$, if $\varepsilon\ge n^{\delta-1/3}$, then
	\begin{align*}
	\bP\left( \|\pml - p\|_1^< \ge  c\varepsilon\right) \le \exp(-c'n^{1-\delta}\varepsilon^2)~. 
	\end{align*} 
\end{corollary}
Corollary \ref{cor:PML_sample_complexity} shows that, over the entire interesting error regime  $\varepsilon\gg n^{-1/3}$ where the classical empirical distribution is sub-optimal, the PML distribution attains the optimal sample complexity $n=\Omega(k/(\varepsilon^2\log k))$. In other words, either the empirical distribution or the PML distribution is minimax rate-optimal in estimating sorted discrete distributions under all parameter configurations. 


\section{Refined Analysis of PML}\label{sec.PML}
This section is devoted to the proof of Theorem \ref{thm:PML_main}. Recall that the error probability $\delta\exp(3\sqrt{n})$ in \cite{acharya2017unified} follows from a union bound over all the possible profiles and the cardinality bound in Lemma \ref{lemma:profile}. To improve over this bound, we provide some useful form of ``continuity" properties of the PML distribution and show that each PML distribution can be approximated by a smaller set of distributions. Hence, using the ``covering" of the smaller set of distributions, a union bound over the covering leads to an improved error probability. Finally, using a chaining argument (or a hierarchy of coverings) as in the literature on uniform concentration inequalities \cite{dudley1967sizes,talagrand2006generic}, we obtain the desired upper bound as in Theorem \ref{thm:PML_main}. 

Toward this direction, we now provide the key continuity lemma for the space of PML estimators. This continuity lemma shows that the set of all PML distributions $p_\phi$ with respect to $\phi\in \Phi_n$ can be approximated (or covered) by a smaller set of distributions up to certain approximation factors that are formally stated next. 
\begin{lemma}[Continuity lemma]\label{lemma:continuity}
Let $\constA\ge 2, c_0\in (0,1)$, $r,s$ be arbitrary constants with $0<s<r\le 1/2$, and $\calM_0\subseteq \calM$ be the set of probability distributions with minimum non-zero probability mass at least $1/(2n^{\constA})$. Then there exists a constant $c=c(\constA,c_0,r,s)>0$ and a subset of probability distributions $\calN \subseteq \calM_0$ such that 
\begin{enumerate}
	\item $|\calN|\le \exp(cn^r\log n)$; 
	\item for any $p\in \calM_0$, there exists some $q\in \calN$ such that for all $S\subseteq \Phi_n$, it holds that
	\begin{align}
	 \nprob{p}{S} &\ge \nprob{q}{S}^{1/(1 - c_0n^{-s})}\exp\left(-cn^{1-2r+s} \right), \label{eq:inequality_p_over_q}\\
	  \nprob{q}{S} &\ge \nprob{p}{S}^{1/(1 - c_0n^{-s})}\exp\left(-cn^{1-2r+s} \right). \label{eq:inequality_q_over_p}
	\end{align} 
\end{enumerate}
\end{lemma}
There are two tuning parameters $r$ and $s$ in Lemma \ref{lemma:continuity}, where the choice of $r$ provides a trade-off between the cardinality of $\calN$ and the multiplicative factors in \eqref{eq:inequality_p_over_q}, \eqref{eq:inequality_q_over_p}, while $s$ corresponds to the trade-off between the exponent of the target probability and the multiplicative factors. Upon choosing an appropriate set of distributions $\calN$, Lemma \ref{lemma:continuity} could be established by a brute-force computation on the likelihood ratio. However, here we adopt a cleaner proof which relies on a Poissonization and the computation of the $\chi^m$-divergence for some large parameter $m$, where $\chi^m(P\|Q) = \bE_Q[(dP/dQ)^m]$ is a generalization of the $\chi^2$-divergence (up to an additive constant $1$). The details are postponed to the appendix, and the usage of the $\chi^m$-divergence is motivated by an analogous covering result under the KL divergence in \cite{acharya2012tight}.

We are ready to bound the failure probability of the PML based estimator and the following lemma will be useful to handle some technical details.
\begin{lemma}[A slight variant of Lemma 4.1 in \cite{charikar2019efficient}]\label{lemma:min_prob}
	Fix a constant $\constA\ge 2$. For any probability distribution $p\in \calM$ and any given profile $\phi\in \Phi_n$, there exists another probability distribution $p'\in \calM$ with minimum non-zero probability mass at least $1/(2n^{\constA})$ such that $p'$ is $(n^{-\constA},3n^{-\constA/2})$-close to $p$ (cf. Definition \ref{def:closeness}), and
	\begin{align*}
	\nprob{p'}{\phi} \ge e^{-6}\cdot \nprob{p}{\phi}.  
	\end{align*}
\end{lemma}

The organization of this section is as follows. In Section \ref{subsec:warmup}, we provide a warmup example using a one-stage covering, which reflects our main ideas and leads to a slightly weaker result compared to Theorem \ref{thm:PML_main}. Then a general chaining argument is provided in Section \ref{subsec:chaining}. 

\subsection{Warmup: One-stage Covering}\label{subsec:warmup}
Here we show how the application of Lemma \ref{lemma:continuity} gives a failure probability of $\delta^{1-c}\exp(c'n^{3/8}\log n)$ for the PML based plug-in estimator. This result reflects the main idea of our proof and already improves over the failure probability of $\delta\exp(3\sqrt{n})$ in \cite{acharya2017unified}. Later in the next subsection, we show that the exponent approaches to $n^{1/3}$ (therefore Theorem \ref{thm:PML_main}) by repeating this procedure several times. 

Fix any probability distribution $p\in \calM$. Define the set of ``good" profiles with respect to distribution $p$ as follows
\begin{align*}
G = \{\phi\in \Phi_n: L(\widehat{T}(\phi), p ) \le \varepsilon \}~,
\end{align*}
these are thee profiles on which the estimator $\widehat{T}$ incurs a small loss under $p$. The following lemma shows that the estimator $\widehat{T}$ also incurs a small loss on all the distributions $q$ that put a large probability mass on $G$. Therefore the value of $d(q, p)$ is small as well.


\begin{lemma}\label{lemma:PML}
Under the assumption of Theorem \ref{thm:PML_main}, if $\nprob{q}{G}>\delta$ for some $q\in \calM$, then $d(q, p)\le 2\varepsilon$. 
\end{lemma}
\begin{proof}
The proof of this result follows from the application of condition \eqref{eq:PML_assumption} and a triangle inequality. We use the technique of proof by contradiction and assume by contradiction that $d(q, p)> 2\varepsilon$.

Recall from the definition of $G$ that $L(\widehat{T}(\phi'), p)\le \varepsilon$ for all $\phi'\in G$. By the compatibility assumption \eqref{eq:compatible_loss} of $L$ we have $d(p,q) \le L(\widehat{T}(\phi'),p) + L(\widehat{T}(\phi'),q)$ and therefore $L(\widehat{T}(\phi'), q)>\varepsilon$ for all $\phi'\in G$. In other words, 
\begin{align*}
\nprob{q}{\{\phi'\in \Phi_n:  L(\widehat{T}(\phi'), q)>\varepsilon\} } \ge \nprob{q}{G} > \delta, 
\end{align*}
a contradiction to the assumption of Theorem \ref{thm:PML_main} and we conclude the proof. 
\end{proof}

For each $\phi\in \Phi_n$, let $p_{\phi}$ be the PML distribution and $p_\phi'$ be the approximate PML distribution returned by Lemma \ref{lemma:min_prob}. By Lemma \ref{lemma:PML}, the following upper bound for the failure probability of PML estimator holds: 
\begin{align}\label{eq:PML_error_prob}
\nprob{p}{\{ \phi\in \Phi_n: d(p_\phi, p)>2\varepsilon + \varepsilon'\}} &\le \nprob{p}{\{ \phi\in \Phi_n: d(p_\phi', p)>2\varepsilon\}} \nonumber \\
&\le \nprob{p}{\Phi_n \backslash G} + \nprob{p}{\{ \phi\in G: d(p_\phi', p)>2\varepsilon\}} \nonumber \\
&\le \delta + \sum_{\phi\in G}\nprob{p}{\phi}\cdot \1(\nprob{p_\phi'}{G} \le \delta ),
\end{align}
where the first inequality follows from  the triangle inequality and the closeness condition given by Lemma \ref{lemma:min_prob}. In the third inequality we use the following assumption of Theorem \ref{thm:PML_main}, $\nprob{p}{\Phi_n \backslash G}\le \delta$. 

Before proceeding to our proof, we first review the proof from \cite{acharya2017unified} that upper bounds the probability in \eqref{eq:PML_error_prob}. For any $\phi\in G$ by the definition of distributions $p_{\phi}$ and $p_{\phi}'$ we have, $$\nprob{p_\phi'}{G} \ge \nprob{p_{\phi}'}{\phi}\ge e^{-6}\nprob{p_{\phi}}{\phi} \ge e^{-6}\nprob{p}{\phi}~.$$
Therefore, 
\begin{align*}
\sum_{\phi\in G}\nprob{p}{\phi}\cdot \1(\nprob{p_\phi'}{G} \le \delta )\le \sum_{\phi\in G}\nprob{p}{\phi}\cdot \1(\nprob{p}{\phi} \le e^6\delta ) \le e^6\delta\cdot |G|\le e^6\delta\cdot \exp(3\sqrt{n}),
\end{align*}
where in the last step we have used Lemma \ref{lemma:profile}. 

We improve the above analysis and, in particular, the potentially loose inequality $\bP(p_\phi',G)\ge \bP(p_\phi',\phi)$. We use our continuity lemma (Lemma \ref{lemma:continuity}) to make this improvement.  In our analysis we work with approximate PML distributions $p_{\phi}'$ returned by Lemma \ref{lemma:min_prob} and further perform a quantization (or discretization) of $p_\phi'$ for each $\phi\in G$. 

For all $\phi$, $p_\phi'\in \calM_0$ and we apply Lemma \ref{lemma:continuity} to construct a distribution $q_{\phi}\in \calN$ such that the inequalities in Lemma \ref{lemma:continuity} hold. We specify the value for parameters $(r,s)$ later. The set $\calN$ only consists of $N\triangleq |\calN|\le \exp(cn^r\log n)$ distributions and we write $\calN = \{q_1, \cdots, q_N\}$, where $q_{i}$ denotes the $i$'th distribution in the set $\calN$. Define,
\begin{align}\label{eq:good_set_small}
G_i = \{\phi\in G: q_\phi = q_i \}, \qquad \forall i\in [N]. 
\end{align}
Using this definition, we upper bound the probability in \eqref{eq:PML_error_prob} as follows,
\begin{align}\label{eq:PML_error_prob_2}
\nprob{p}{\{ \phi\in \Phi_n: d(p_\phi, p)>2\varepsilon + \varepsilon'\}} \le \delta + \sum_{i=1}^N \nprob{p}{G_i}\cdot \1(\exists \phi\in G_i: \nprob{p_{\phi}'}{G_i}\le \delta). 
\end{align}
We now show that if $\bP(p_\phi', G_i)$ is small for some $\phi\in G_i$, then $\nprob{p}{G_i}$ must be small as well and we use the quantized distribution $q_i$ to relate the probability values. Consider any $\phi_0\in G_i$, let $p_{\phi_{0}}$ be the PML distribution and $p_{\phi_{0}}'$ be the approximate PML distribution returned by Lemma \ref{lemma:min_prob}. Apply Lemma \ref{lemma:continuity} to $p_{\phi_{0}}'$, let $q_{\phi_{0}} \in \calN$ be the quantized distribution. As $\phi_0\in G_i$, we have $q_{\phi_{0}}=q_{i}$ and the following chain of inequalities holds:
\begin{align*}
\nprob{p_{\phi_0}'}{G_i} &\stepa{\ge} \nprob{q_i}{G_i}^{1/(1-c_0n^{-s})}\exp\left(-cn^{1-2r+s} \right) \\
&= \Big(\sum_{\phi\in G_i}  \nprob{q_i}{\phi} \Big)^{1/(1-c_0n^{-s})}\exp\left(-cn^{1-2r+s} \right) \\
&\stepb{\ge}  \Big(\sum_{\phi\in G_i}  \nprob{p_{\phi}'}{\phi}^{1/(1-c_0n^{-s})} \exp\left(-cn^{1-2r+s} \right) \Big)^{1/(1-c_0n^{-s})}\exp\left(-cn^{1-2r+s} \right) \\
&\stepc{\ge} \Big(\sum_{\phi\in G_i}  [e^{-6}\nprob{p_{\phi}}{\phi}]^{1/(1-c_0n^{-s})} \Big)^{1/(1-c_0n^{-s})}\exp\Big(-cn^{1-2r+s}-\frac{cn^{1-2r+s}}{1-c_0n^{-s}} \Big) \\
&\stepd{\ge} \Big(\sum_{\phi\in G_i}  \nprob{p}{\phi}^{1/(1-c_0n^{-s})} \Big)^{1/(1-c_0n^{-s})}\exp\Big(-cn^{1-2r+s}-\frac{cn^{1-2r+s}}{1-c_0n^{-s}}-\frac{6}{1-c_0n^{-s}} \Big) \\
&\stepe{\ge} \Big(\nprob{p}{G_i}^{1/(1-c_0n^{-s})} |G_i|^{-c_0/(n^s-c_0)} \Big)^{1/(1-c_0n^{-s})}\exp(-\csto n^{1-2r+s}),
\end{align*}
where (a) follows from \eqref{eq:inequality_p_over_q} in Lemma \ref{lemma:continuity}, (b) follows from \eqref{eq:inequality_q_over_p} in Lemma \ref{lemma:continuity}, (c) is due to the approximate PML property from Lemma \ref{lemma:min_prob}, (d) is due to the definition of the PML distribution. In the final inequality (e), we used Jensen's inequality (for the first term) applied to the convex function $t\mapsto t^{1/(1-c_0n^{-s})}$ for $t\ge 0$, and $c_1>0$ is some absolute constant depending on $(c,c_0)$. 

Using the cardinality bound $|G_i|\le \exp(3\sqrt{n})$ (Lemma \ref{lemma:profile}) and rearranging terms we get that for all $\phi_0 \in G_{i}$ there exists a absolute constant $\cstt$ such that the following inequality holds,
\begin{align*}
\nprob{p}{G_i} \le \nprob{p_{\phi_0}'}{G_i}^{(1-c_0n^{-s})^2}\cdot \exp\left(\cstt\left(n^{1-2r+s} + n^{1/2-s}\right) \right),
\end{align*}
where in the above inequality we also used $1/(1-c_0n^{-s}) \in O(1)$. The above inequality holds for all $\phi_0 \in G_{i}$ and if there exists a profile $\phi \in G_i$ such that $\bP(p_\phi', G_i)\le \delta$, then
\begin{equation}\label{eq:bound}
\nprob{p}{G_i} \le \delta^{(1-c_0n^{-s})^2}\cdot \exp\left(\cstt\left(n^{1-2r+s} + n^{1/2-s}\right) \right)~.
\end{equation}
Consequently substituting \eqref{eq:bound} into \eqref{eq:PML_error_prob_2} and using the cardinality bound $N\le \exp(cn^r\log n)$, we conclude that
\begin{align*}
\nprob{p}{\{ \phi\in \Phi_n: d(p_\phi, p)>2\varepsilon + \varepsilon'\}} &\le \delta + N\cdot  \delta^{(1-c_0n^{-s})^2}\cdot \exp\left(\cstt\left(n^{1-2r+s} + n^{1/2-s}\right) \right) \\
&\le \delta +  \delta^{(1-c_0n^{-s})^2}\cdot \exp\left(\cstth \left(n^r\log n+ n^{1-2r+s} + n^{1/2-s}\right) \right), 
\end{align*}
where $c_3>0$ is some absolute constant. Further choosing $r = 3/8, s= 1/8$ to balance the terms on the above exponent, and the constant $c_0>0$ to be small enough, we conclude that
\begin{align*}
\nprob{p}{\{ \phi\in \Phi_n: d(p_\phi, p)>2\varepsilon + \varepsilon'\}} \le \delta + \delta^{1-c}\exp\left(\cstth n^{3/8}\log n\right)
\end{align*}
for any prescribed parameter $c>0$. In other words, using a one-stage covering argument the failure probability of PML estimator improves in the exponent from $O(\sqrt{n})$ to $O(n^{3/8}\log n)$.

\subsection{A General Chaining Argument: Proof of Theorem \ref{thm:PML_main}}\label{subsec:chaining}
Here we provide the proof of Theorem \ref{thm:PML_main} by applying the covering arguments in the previous section several times. Specifically, for each profile $\phi\in \Phi_n$ we successively quantize the PML distribution $p_\phi$ into $M$ levels, where $M$ is the smallest integer solution to the inequality
\begin{align}\label{eq:M_equation}
\frac{1}{12(3\cdot 2^{M-1} - 1)} < c,
\end{align}
where $c>0$ is the constant in the statement of Theorem \ref{thm:PML_main}. 

\begin{figure}[t]
	\centering
	\begin{tikzpicture}
	\draw[color = red, thick] (0,0) ellipse (6cm and 2cm); 
	\draw[color = blue, dashed, thick] (-2.5,0) ellipse (2cm and 1.5cm); 
	\draw[color = blue, dashed, thick] (2.5,0) ellipse (2cm and 1.5cm); 
	\draw[color = cyan, rotate around={45:(-3.5,0.5)}] (-3.5,0.5) ellipse (1cm and 0.4cm);
	\draw[color = cyan] (-2.5, -1) ellipse (1cm and 0.2cm);
	\draw[color = cyan, rotate around={-30:(-1.5,0.5)}] (-1.5,0.5) ellipse (1cm and 0.3cm);
	\draw[color = cyan] (-2.5, -0.2) circle (0.2cm);
	\draw[color = cyan, rotate around={60:(1.5,0.3)}] (1.5,0.3) ellipse (1.1cm and 0.5cm);
	\draw[color = cyan, rotate around={150:(3.4,0.6)}] (3.4,0.6) ellipse (1.1cm and 0.5cm);
	\draw[color = cyan, rotate around={105:(2.8,-0.6)}] (2.8,-0.6) ellipse (0.8cm and 0.3cm);
	\draw[fill] (1.9, 1.2) circle (0.05cm); \draw[fill] (2.7, 0.9) circle (0.05cm);
	\draw[fill] (3.5, 1) circle (0.05cm); \draw[fill] (3.9, 0.1) circle (0.05cm);
	\draw[fill] (1, -0.5) circle (0.05cm); \draw[fill] (1.7, 0) circle (0.05cm);
	\draw[fill] (1.1, 0.3) circle (0.05cm); \draw[fill] (3, -1.2) circle (0.05cm);
	\draw[fill] (2.7, -0) circle (0.05cm); \draw[fill] (-2.5, -0.2) circle (0.05cm);
	\draw[fill] (-2.5, -1) circle (0.05cm); \draw[fill] (-4, -0.1) circle (0.05cm);
	\draw[fill] (-3, 1) circle (0.05cm); \draw[fill] (-3.2, 0.4) circle (0.05cm);
	\draw[fill] (-2.2, 0.8) circle (0.05cm); \draw[fill] (-0.8, 0.1) circle (0.05cm);
	\draw[fill] (-3, -1) circle (0.05cm); \draw[fill] (-2, -1) circle (0.05cm);
	\node[red] at (3,-2) {$G$}; 
	\node[blue] at (-0.8,-1.4) {$G_{1,2}$}; \node[blue] at (4.5,-1) {$G_{2,2}$}; 
	\node[cyan] at (-2.5,0.2) {$G_{2,1}$};\node[cyan] at (-3.5,-0.5) {$G_{1,1}$}; 
	\node[cyan] at (-1.8, -0.7) {$G_{3,1}$};\node[cyan] at (-1.5,-0.1) {$G_{4,1}$}; 
	\node[cyan] at (1.8, -0.7) {$G_{5,1}$};\node[cyan] at (2.2,-1.2) {$G_{6,1}$}; 
	\node[cyan] at (3.8, -0.4) {$G_{7,1}$}; \node at (1.3, 0.3) {$\phi$};
	\end{tikzpicture}
	\caption{A pictorial illustration of the chain of coverings when $M=2$. Profiles are represented by dots, and there are two levels of coverings of good profiles $G$ represented by solid and dashed curves, respectively, where high-level coverings $\{G_{1,2}, G_{2,2}\}$ are coarser than low-level ones $\{G_{1,1},\cdots,G_{7,1}\}$. Each profile $\phi\in G$ has a label in each covering, e.g. $\phi \in G_{5,1}\subseteq G_{2,2}$.}\label{fig:chain}
\end{figure}
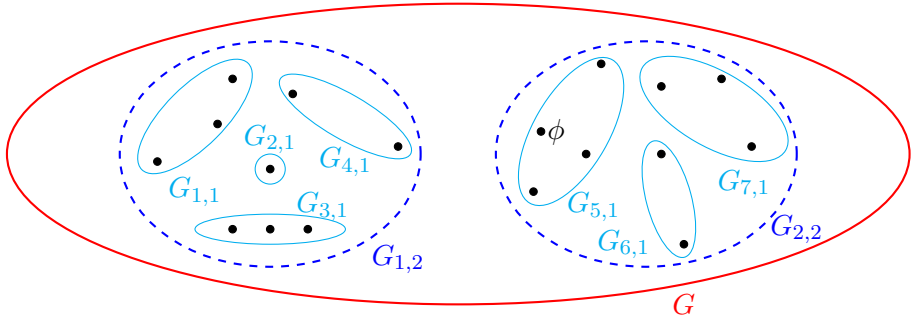

The complete details of the proof are as follows. As in the previous section, define the set of ``good" profiles $G$ and let $p_{\phi}$ and $p_\phi'$ be the PML and approximate PML distributions. The upper bound \eqref{eq:PML_error_prob} on the failure probability of the PML estimator again holds and we restate it here for convenience: 
\begin{align}\label{eq:PML_error_prob_gc}
\nprob{p}{\{ \phi\in \Phi_n: d(p_\phi, p)>2\varepsilon + \varepsilon'\}} \le \delta + \sum_{\phi\in G}\nprob{p}{\phi}\cdot \1(\nprob{p_\phi'}{G} \le \delta )~.
\end{align}
We invoke Lemma \ref{lemma:continuity} $M$ times and pick $M$ subsets of distributions $\calN_1, \cdots, \calN_M$, with $N_m \triangleq |\calN_m| \le \exp(c_mn^{r_m}\log n)$. Further inequalities \eqref{eq:inequality_p_over_q} and \eqref{eq:inequality_q_over_p} hold with parameters $(r_m, s_m)$. The specific choices of $(r_m, s_m)$ are as follows: for $m\in [M]$, 
\begin{align}\label{eq:r_m}
r_m &= \frac{1}{3}\left(1 + \frac{1}{2^{m+1}}\right) - \frac{1}{6(3\cdot 2^{M-1}-1 )}\left(1-\frac{1}{2^m}\right), \\ \label{eq:s_m}
s_m &= \frac{1}{3\cdot 2^m} - \frac{1}{12(3\cdot 2^{M-1}-1)}\left(3-\frac{1}{2^{m-2}}\right). 
\end{align}
After some algebra, we have $5/12 > r_1 > r_2 > \cdots > r_M > 1/3$, and $1/6 > s_1 > s_2 > \cdots > s_M > 0$. The above expressions \eqref{eq:r_m} and \eqref{eq:s_m} are chosen to satisfy the following linear equations, 
\begin{align}\label{eq:covering_identity}
\begin{split}
1 - 2r_m + s_m &= t, \qquad \forall m\in [M], \\
r_{m-1} - s_m &= t, \qquad \forall m\in [M+1], 
\end{split}
\end{align}
with the convention $r_0 = 1/2, s_{M+1} = 0$, and $t = 1/3 + 1/[12(3\cdot 2^{M-1}-1)] < 1/3 + c$ thanks to the choice of $M$ in \eqref{eq:M_equation}. 

Based on the subsets of distributions $\calN_1, \cdots, \calN_M$, we quantize each approximate PML distribution $p_{\phi}'$ as follows. Let $q_{\phi,0} \triangleq p_{\phi}'$, and for any $m\in [M]$, let $q_{\phi, m}\in \calN_m$ be the quantization of $q_{\phi,m-1}$ to the subset $\calN_m$. For any $m\in [M]$, the distributions $({q}_{\phi, m-1}, {q}_{\phi,m})$ satisfy the inequalities \eqref{eq:inequality_p_over_q}, \eqref{eq:inequality_q_over_p} with parameters $(r_m, s_m)$, that is for all $S\subseteq \Phi_n$ the following inequality holds: 
\begin{align}
\nprob{q_{\phi,m-1}}{S} &\ge \nprob{q_{\phi,m}}{S}^{1/(1 - c_0n^{-s_m})}\exp\left(-c_mn^{1-2r_m+s_m} \right), \label{eq:inequality_p_over_q_new}\\
\nprob{q_{\phi,m}}{S} &\ge \nprob{q_{\phi,m-1}}{S}^{1/(1 - c_0n^{-s_m})}\exp\left(-c_mn^{1-2r_m+s_m} \right). \label{eq:inequality_q_over_p_new}
\end{align}
where $c_0>0$ and $c_m>0$ for each $m \in [M]$ are absolute constants. In the above process $p_\phi'\in \calM_0$ and the output distribution of Lemma \ref{lemma:continuity} always belongs to $\calM_0$, therefore it is feasible to reapply Lemma \ref{lemma:continuity} to all the quantized distributions. Therefore for each $\phi\in G$, we have the following chain of quantized distributions $p_{\phi}'=q_{\phi,0} \to q_{\phi,1}\to q_{\phi,2}\to \cdots \to q_{\phi,M}$ and for each $m\in [0,M]$ it induces the following covering of $G$. For each $m \in [0,M]$, let $\calN_m=\{q_{1,m}, \cdots, q_{N_m,m}\}$ and we use $q_{i,m}$ to denote its $i$'th element. Define 
\begin{align}\label{eq:covering}
G_{i,m} = \{\phi\in G: {q}_{\phi,m} = {q}_{i,m} \}, \qquad \forall i\in [N_m].
\end{align}
Clearly, the sets $\{G_{i,m}\}_{i\in [N_m]}$ form a partition of $G$. At different levels, the above partitions form a chain in the sense that $G_{m,i}$ is a disjoint union of the sets $G_{m-1,j}$, where $j\in [N_{m-1}]$ are such that $q_{j,m-1}\in \calN_{m-1}$ quantizes to $q_{i,m}\in \calN_m$. In other words, there is a hierarchy of partitions where high-level partitions are coarser than low-level ones. A pictorial illustration of the above chain of partitions (or coverings) for $M=2$ is illustrated in Figure \ref{fig:chain}.

Now using the same argument as \eqref{eq:PML_error_prob_2} but applied to set $\calN_M$, we get
\begin{align}\label{eq:PML_error_prob_3}
\nprob{p}{\{ \phi\in \Phi_n: d(p_\phi, p)>2\varepsilon + \varepsilon'\}} \le \delta + \sum_{i=1}^{N_M} \nprob{p}{G_{i,M}}\cdot \1(\exists \phi\in G_{i,M}: \nprob{p_\phi'}{G_{i,M}}\le \delta). 
\end{align}

For each $i\in [N_M]$, consider any profile $\phi\in G_{i,M}$. Let $i_m(\phi)\in [N_m]$ be the index of the set in the $m$-th level partition to which the profile $\phi$ belongs, i.e. $\phi \in G_{i_m(\phi),m}$ and $i_M(\phi) = i$. Applying the inequality \eqref{eq:inequality_p_over_q_new} (i.e. Lemma \ref{lemma:continuity}) repeatedly to the chain of quantized distributions $p_\phi' \to q_{i_1(\phi),1}\to q_{i_2(\phi),2}$ leads to,
\begin{align*}
\nprob{p_\phi'}{G_{i,M}} &\ge \nprob{q_{i_1(\phi),1}}{G_{i,M}}^{1/(1-c_0n^{-s_1})}\exp\left(-c_1n^{1-2r_1+s_1}\right) \\
&\ge \nprob{q_{i_2(\phi),2}}{G_{i,M}}^{1/[(1-c_0n^{-s_1})(1-c_0n^{-s_2})]}\exp\left(-c_1n^{1-2r_1+s_1}-\frac{c_2n^{1-2r_2+s_2}}{1-c_0n^{-s_1}}\right)~.
\end{align*}
Further continuing to apply inequality \eqref{eq:inequality_p_over_q_new} to the remaining chain of quantized distributions $q_{i_2(\phi),2}\to q_{i_3(\phi),3}\to \cdots\to q_{i,M}$ we get,
\begin{align*}
\nprob{p_\phi'}{G_{i,M}}\ge \nprob{q_{i,M}}{G_{i,M}}^{\prod_{m=1}^M (1-c_0n^{-s_m})^{-1}}\exp\left(-\sum_{m=1}^M \frac{c_m n^{1-2r_m+s_m}}{\prod_{m'<m} (1-c_0n^{-s_{m'}}) } \right).
\end{align*}
Therefore for each $i\in [N_M]$, if there exists a profile $\phi\in G_{i,M}$ such that $\bP(p_\phi', G_{i,M})\le \delta$, then
\begin{align*}
\delta \ge \nprob{q_{i,M}}{G_{i,M}}^{\prod_{m=1}^M (1-c_0n^{-s_m})^{-1}}\exp\left(-\sum_{m=1}^M \frac{c_m n^{1-2r_m+s_m}}{\prod_{m'<m} (1-c_0n^{-s_{m'}}) } \right).
\end{align*}
By choosing a constant $c_0>0$ small enough such that $\prod_{m=1}^M (1-c_0n^{-s_m})\ge 1-c/2$ and rearranging terms, the above inequality leads to
\begin{align}\label{eq:going_down}
\nprob{q_{i,M}}{G_{i,M}} \le \delta^{1-c/2} \exp\Big(C\sum_{m=1}^M n^{1-2r_m+s_m}\Big).
\end{align}
We interpret \eqref{eq:going_down} as the ``going-down" process, where we translate the small probability event $\bP({p_\phi'},{G_{i,M}})$ under the approximate PML distribution into a small probability event $\nprob{q_{i,M}}{G_{i,M}}$ under a quantized distribution that belongs to a much smaller set. 

We now describe the ``going-up" process, where the target is to show that a small probability event $\nprob{q_{i,M}}{G_{i,M}}$ under the quantized distribution also implies a small probability event $\nprob{p}{G_{i,M}}$ under the true distribution $p$. To this end, we use the other inequality \eqref{eq:inequality_q_over_p_new} (i.e. Lemma \ref{lemma:continuity}) for each $m\in [M]$ and $i \in [N_M]$ to get, 
\begin{align}\label{eq:going_up_induction}
&\sum_{j\in [N_m]: G_{j,m}\subseteq G_{i,M}} \nprob{q_{j,m}}{G_{j,m}} = \sum_{j\in [N_m]: G_{j,m}\subseteq G_{i,M}}~~ \sum_{j'\in [N_{m-1}]: G_{j',m-1}\subseteq G_{j,m} } \nprob{q_{j,m}}{G_{j',m-1}} \nonumber\\
&\stepa{\ge}  \sum_{j\in [N_m]: G_{j,m}\subseteq G_{i,M}}~~ \sum_{j'\in [N_{m-1}]: G_{j',m-1}\subseteq G_{j,m} } \nprob{q_{j',m-1}}{G_{j',m-1}}^{\frac{1}{1-c_0n^{-s_m}}}\exp\left(-c_mn^{1-2r_m+s_m}\right) \nonumber\\
&\stepb{\ge} \Big(\sum_{j\in [N_m]: G_{j,m}\subseteq G_{i,M}}~~ \sum_{j'\in [N_{m-1}]: G_{j',m-1}\subseteq G_{j,m} } \nprob{q_{j',m-1}}{G_{j',m-1}}\Big)^{\frac{1}{1-c_0n^{-s_m}}}\frac{\exp\left(-c_mn^{1-2r_m+s_m}\right)}{(N_{m-1})^{c_0/(n^{s_m} - c_0)}}\nonumber\\
&\stepc{\ge} \big(\sum_{j'\in [N_{m-1}]: G_{j',m-1}\subseteq G_{i,M}} \nprob{q_{j',m-1}}{G_{j',m-1}} \big)^{\frac{1}{1-c_0n^{-s_m}}} \exp\left(-c_m n^{1-2r_m+s_m}-\frac{c_{m-1}n^{r_{m-1}}\log n}{c_0^{-1}n^{s_m} - 1} \right)
\end{align}
where (a) is due to inequality \eqref{eq:inequality_q_over_p_new} (i.e. Lemma \ref{lemma:continuity}), (b) follows from the convexity of $t\mapsto t^{1/(1-c_0n^{-s_m})}$ for $t\ge 0$ and the fact that the number of summands is at most $N_{m-1}$, and (c) follows from the rearrangement of the sum and the cardinality bound $N_{m-1}\le \exp(c_{m-1} n^{r_{m-1}}\log n)$ in Lemma \ref{lemma:continuity}. Note that when $m=1$, in the above inequality we have $q_{\phi,0} = p_{\phi}'$, $G_{\phi,0} = \{\phi\}$, and $N_0=|G|$ with $r_0 = 1/2$ (which follows from Lemma \ref{lemma:profile}). Hence, applying \eqref{eq:going_up_induction} repeatedly with $m=M, M-1,\cdots,1$, for the above choice of $c_0$ it holds that
\begin{align}\label{eq:going_up}
&\nprob{q_{i,M}}{G_{i,M}}\nonumber \\ 
&\ge \Big(\sum_{\phi\in G_{i,M}} \nprob{p_{\phi}'}{\phi} \Big)^{\prod_{m=1}^{M}\frac{1}{(1-c_0n^{-s_m})}} \exp\left(-C\sum_{m=1}^M \left(n^{1-2r_m+s_m} + n^{r_{m-1}-s_m}\log n\right) \right) \nonumber \\
&\stepa{\ge} \Big(\sum_{\phi\in G_{i,M}} \nprob{p_{\phi}'}{\phi}\Big)^{1/(1-c/2)}\exp\left(-C\sum_{m=1}^M \left(n^{1-2r_m+s_m} + n^{r_{m-1}-s_m}\log n\right) \right) \nonumber \\
&\stepb{\ge} \Big(\sum_{\phi\in G_{i,M}} e^{-6}\nprob{p_\phi}{\phi}\Big)^{1/(1-c/2)}\exp\left(-C\sum_{m=1}^M \left(n^{1-2r_m+s_m} + n^{r_{m-1}-s_m}\log n\right) \right) \nonumber \\
&\stepc{\ge} \Big(\sum_{\phi\in G_{i,M}} e^{-6}\nprob{p}{\phi}\Big)^{1/(1-c/2)}\exp\left(-C\sum_{m=1}^M \left(n^{1-2r_m+s_m} + n^{r_{m-1}-s_m}\log n\right) \right) \nonumber\\
&= (e^{-6}\nprob{p}{G_{i,M}})^{1/(1-c/2)} \exp\left(-C\sum_{m=1}^M \left(n^{1-2r_m+s_m} + n^{r_{m-1}-s_m}\log n\right) \right),
\end{align}
where (a) follows because $\prod_{m=1}^{M} (1-c_0n^{-s_m}) \geq 1-c/2$ by the choice of $c_0$, (b) follows from the definition of the approximate PML in Lemma \ref{lemma:min_prob}, and (c) follows from the definition of PML that $\nprob{p_{\phi}}{\phi}\ge \nprob{p}{\phi}$.

Combining \eqref{eq:going_down} and \eqref{eq:going_up}, we conclude that whenever there exists a profile $\phi\in G_{i,M}$ such that $\bP(p_\phi',G_{i,M})\le \delta$, then the following inequality holds: 
\begin{align*}
\nprob{p}{G_{i,M}} \le \delta^{1-c}\cdot \exp\left(C\sum_{m=1}^M \left(n^{1-2r_m+s_m} + n^{r_{m-1}-s_m}\log n\right) \right). 
\end{align*}
Therefore substituting the above inequality in \eqref{eq:PML_error_prob_3} we get the following upper bound on the failure probability of PML estimator,
\begin{align}\label{eq:PML_error_prob_final}
&\nprob{p}{\{ \phi\in \Phi_n: d(p_\phi, p)>2\varepsilon + \varepsilon'\}} \nonumber \\
& \quad \quad \le \delta + \delta^{1-c}\cdot \exp\left(Cn^{r_M}\log n+ C\sum_{m=1}^M \left(n^{1-2r_m+s_m} + n^{r_{m-1}-s_m}\log n\right) \right).
\end{align}
Finally, as the equation \eqref{eq:covering_identity} is satisfied by the parameters $\{(r_m,s_m)\}_{m\in[M]}$ with $t < 1/3 + c$, the inequality \eqref{eq:PML_error_prob_final} gives the claimed result of Theorem \ref{thm:PML_main}.

\section{Estimator Construction of Sorted Distribution}\label{sec.estimator}
In this section, we construct the desired estimator for distribution estimation under Wasserstein distance (Theorem \ref{thm:sorted_main}). This section is organized as follows: an explicit estimator is constructed in Section \ref{subsec.estimator}, and Section \ref{subsec.analysis} provides a high-level road map of the analysis of the previous estimator and proves Theorem \ref{thm:sorted_main}, assuming several technical lemmas whose proofs are relegated to the appendix. Further discussions are placed in Section \ref{subsec.poisson_approx}, and in particular as a by-product, we show a better tradeoff between the polynomial approximation error and the maximum magnitudes of polynomial coefficients for the Poisson approximation of general $1$-Lipschitz functions, which is a pure approximation-theoretic problem and might be of independent interest. 

\subsection{Estimator Construction}\label{subsec.estimator}
Our estimator generalizes the idea of the \emph{local moment matching} estimator in \cite{han2018local}. Specifically, to obtain a small Wasserstein distance $\text{W}_1(\widehat{\mu},\mu_p)$, we partition the support $[0,1]$ of $\mu_p$ into several local intervals, and require that the restriction of measure $\widehat{\mu}$ on each local interval have matched low-order moments (up to order $\Theta(\log n)$) as those of $\mu_p$. In \cite{han2018local}, sample splitting is necessary to locate each symbol in the correct local interval, while incorrect identification may occur with probability inverse polynomial in $n$. Also, one need to solve a feasibility program in each local interval, where each program may fail to have a feasible solution with probability inverse polynomial in $n$. Hence, although the estimator in \cite{han2018local} is minimax rate-optimal in expectation, it is unstable in the sense that the failure probability cannot be super-polynomially small. To address this issue, we stabilize the estimator by removing the sample splitting and considering a single linear program taking all local intervals into account simultaneously. Specifically, our estimator is constructed as follows: 
\begin{enumerate}
	\item Split $[0,1]$ into local intervals: let $c_1>0$ be a large tuning constant to be specified later, and assume that $M\triangleq n/(c_1\log n)$ is an integer. For $m\in [M]$, let $I_m$ be the local interval
	\begin{align}\label{eq:local_interval}
	 I_m = \left(\frac{c_1\log n}{n}\cdot (m-1)^2, \frac{c_1\log n}{n}\cdot m^2 \right].  
	\end{align}
	We also define the following slightly enlarged intervals of $I_m$: 
	\begin{align}\label{eq:local_interval_enlarged}
	\widetilde{I}_m = \left[\frac{c_1\log n}{n}\cdot \left(m-\frac{5}{4}\right)_+^2, \frac{c_1\log n}{n}\cdot \left(m+\frac{1}{4}\right)^2 \right], 
	\end{align}
	with $(x)_+ \triangleq \max\{x,0\}$. Let $x_1 =0$, and $x_m$ be the mid-point of the interval $I_m$ for $m\ge 2$. Let $\widetilde{\ell}_m$ be the length of the interval $\widetilde{I}_m$. 
	\item Estimators for smoothed moments: for each $d\in \NN$ and $m\in [M]$, define a  function $g_{d,x_m}: \bR_+ \to \bR$ with $g_{0,x_m}(x)\equiv 1$, and
	\begin{align}\label{eq:g_function}
	g_{d,x_m}(x) \triangleq \sum_{d'=0}^d \binom{d}{d'}(-x_m)^{d-d'}\prod_{d''=0}^{d'-1}\left(x-\frac{2d''}{n}\right).
	\end{align}
	We also define the following modified version that cuts off the function $g$ near the boundary of the $m$-th local interval: let $x_{m,\text{L}} = c_1(m-3/2)_+^2\log n/n, x_{m,\text{R}}=c_1(m+1/2)^2\log n/n$, and
	\begin{align}\label{eq:tilde_g_function}
	 \widetilde{g}_{d,x_m}(x) = \begin{cases}
	  g_{d,x_m}(x_{m,\text{L}}) & \text{if } x\le x_{m,\text{L}}, \\
	  g_{d,x_m}(x) & \text{if } x_{m,\text{L}} < x < x_{m,\text{R}}, \\
	  g_{d,x_m}(x_{m,\text{R}}) & \text{if } x\ge x_{m,\text{R}}. 
	 \end{cases}
	\end{align}
	Now for each $d\in \NN$ and $m\in [M]$, we define our estimator of the smoothed $d$-th moments in the $m$-th local interval as
	\begin{align}\label{eq:est_smoothed_moments}
	\widehat{M}_{m,d} \triangleq \sum_{j=1}^k \sum_{s\in nI_m/2}\bP\left(\mathsf{B}\left(h_j, \frac{1}{2}\right) = s \right)\cdot \widetilde{g}_{d,x_m}\left(\frac{h_j - s}{n/2}\right),
	\end{align}
	where $h_j = \sum_{i=1}^n \1(X_i = j)$ denotes the total number of occurrences (i.e. histogram) of the symbol $j$ in the observations (cf. Definition \ref{def:histogram}), and $aI\triangleq \{ax: x\in I\}$. 
	\item The final linear program: let $c_2>0$ be a small tuning constant to be specified later, and assume that $D\triangleq c_2\log n$ is an integer. Define the initial estimator $\widehat{\mu}_0$ to be any solution of the following linear program: 
	\begin{align}\label{eq:linear_program}
	\begin{split}
	\text{minimize}\qquad &L(\widehat{\mu}_0, \widehat{M}) \triangleq \sum_{m=1}^M \widetilde{\ell}_m\left(\sum_{d=1}^D \widetilde{\ell}_m^{-d}\left|k\cdot \int_{\widetilde{I}_m} (x-x_m)^d\widehat{\mu}_m(dx) - \widehat{M}_{m,d} \right| \right. \\
	& \left.  \qquad \qquad \qquad \qquad \qquad + \left|\sum_{m'\ge m}\left(k\cdot \widehat{\mu}_{m'}(\widetilde{I}_{m'}) - \widehat{M}_{m',0} \right) \right| \right) \\
	\text{subject to}\qquad & \widehat{\mu}_0 = \sum_{m=1}^M \widehat{\mu}_m, \quad \text{supp}(\widehat{\mu}_m) \subseteq \widetilde{I}_m, \quad \widehat{\mu}_0(\bR)\le 1 , \quad \int_{[0,1]} x\widehat{\mu}_0(dx) \le k^{-1}. 
	\end{split}
	\end{align}
	The final estimator $\widehat{\mu}$ is constructed to be $\widehat{\mu} = \widehat{\mu}_0 + (1-\widehat{\mu}_0(\bR))\delta_0 $, with $\delta_0$ being the Dirac measure on the single point $0$. 
\end{enumerate}

A few remarks of the above construction are in order: 
\begin{itemize}
	\item The high-level idea: we provide some intuition on why the estimator $\widehat{\mu}$ in \eqref{eq:linear_program} satisfies the desired concentration inequality in Theorem \ref{thm:sorted_main}. In fact, applying the dual representation of the Wasserstein-$1$ distance (cf. Lemma \ref{lemma.dual}), invoking Jackson-type theorems in approximation theory (cf. Lemma \ref{lemma:jackson}), and using the polynomial coefficient bounds (cf. Lemma \ref{lemma:coefficients}) will yield to the following deterministic inequality: 
	\begin{align*}
	k\cdot \text{W}_1(\widehat{\mu}, \mu_p) \le C \left( \sqrt{\frac{k}{n\log n}} + n^{O(c_2)}L(\widehat{\mu}, \mu_p) \right),
	\end{align*}
	where $C>0$ is some absolute constant, and with a slight abuse of notation, we write $L(\widehat{\mu}, \mu_p)$ for the quantity $L(\widehat{\mu}, M)$ in \eqref{eq:linear_program} with true moments $M_{m,d} = k\cdot \int_{\widetilde{I}_m} (x-x_m)^d \mu_m(dx)$, with a suitable decomposition $\mu_p = \sum_{m=1}^M \mu_m$. In other words, our choice of $L$ in \eqref{eq:linear_program} serves as a suitable upper bound to the target Wasserstein distance. 
	
	To proceed, we make a crucial observation that (a slight variant of) $\mu_p$ is always a feasible solution of \eqref{eq:linear_program}, and therefore the definition of the minimizer $\widehat{\mu}_0$ gives
	\begin{align*}
	L(\widehat{\mu}, \mu_p) = L(\widehat{\mu}_0, \mu_p) \le L(\widehat{\mu}_0, \widehat{M}) + L(\mu_p, \widehat{M}) \le 2L(\mu_p, \widehat{M}), 
	\end{align*}
	where the final quantity $L(\mu_p, \widehat{M})$ quantifies the estimation performance of the local moment estimators $\widehat{M}_{m,d}$ for the true local moments. The proof of Theorem \ref{thm:sorted_main} will be completed by showing that $\bE[L(\mu_p, \widehat{M})] = \widetilde{O}(n^{-1/3})$, and changing a single observation will only result in an $\widetilde{O}(n^{-1})$ change on $L(\mu_p, \widehat{M})$. 
	\item The choice of the local intervals: the local intervals $I_m$ (and the enlarged variants $\widetilde{I}_m$) are chosen so that symbols with probability in $I_m$ and in $[0,1]\backslash \widetilde{I}_m$ essentially do not affect each other in the sense of Lemma \ref{lemma.localization}. These choices coincide with ``confidence set" in \cite{han2016minimax}, and the exact ones appear in \cite{han2018local,hao2019unified,hao2020profile}. 
	\item The function $g_{d,x_m}(x)$ and Poissonization: the function $g_{d,x_m}(x)$ in \eqref{eq:g_function} is chosen so that if $h\sim \mathsf{Poi}(np/2)$, we have \cite[Example 2.8]{Withers1987}
	\begin{align*}
	\bE\left[g_{d,x_m}\left(\frac{h}{n/2}\right) \right] = (p - x_m)^d. 
	\end{align*}
	Note that here we assume that each histogram $h_j$ follows a Poisson distribution instead of Binomial in the i.i.d. sampling model. In fact, we will analyze the performance of $\widehat{\mu}$ in the Poissonized model: the sample size is an independent Poisson random variable $N\sim \mathsf{Poi}(n)$ instead of $n$, so that the counts $(h_1,\cdots,h_k)$ are mutually independent. However, we remark that the same result also holds for the estimator in \eqref{eq:linear_program} without any modification under the original i.i.d. sampling model, by using the reduction technique at the end of this section. 
	\item The modification $\widetilde{g}_{d,x_m}(x)$: we cut off the function $g_{d,x_m}(x)$ outside the interval $[x_{m,\text{L}}, x_{m,\text{R}}]$ to ensure a uniform small difference $|\widetilde{g}_{d,x_m}(x) - \widetilde{g}_{d,x_m}(y)|$ for $|x-y|\le 2/n$, which helps to establish the desired concentration result. The choice of the interval $[x_{m,\text{L}}, x_{m,\text{R}}]$ is slightly larger than $\widetilde{I}_m$, for the function $g_{d,x_m}$ will be essentially defined on $\widetilde{I}_m$. 
	\item Sample splitting and smoothed moments $\widehat{M}_{m,d}$: ideally we would like to estimate the local moments $\sum_{j=1}^k p_j^d\cdot \1(p_j\in I_m)$ for each $m\in [M]$ and $d\in \NN$. However, the above quantity does not admit an unbiased estimator, and is information-theoretically hard to estimate for $p_j$ close to the boundary of $I_m$. In previous works of functional estimation, the idea of sample splitting was used to overcome the above difficulty: split the empirical probability $\widehat{p}_j$ into two independent halves $\widehat{p}_{j,1}$ and $\widehat{p}_{j,2}$, then the modified moments $\sum_{j=1}^k p_j^d\cdot \1(\widehat{p}_{j,1}\in I_m)$ have a simple unbiased estimator based on $\widehat{p}_{j,2}$. However, as we argued at the beginning of this section, sample splitting will make the overall estimator unstable. To overcome this difficulty, we consider the following \emph{smoothed moments}
	\begin{align}\label{eq:smoothed_moments}
	 M_{m,d} = \sum_{j=1}^k (p_j-x_m)^d \cdot \bP\left(\mathsf{Poi}\left(\frac{np_j}{2}\right) \in \frac{n}{2}I_m \right),
	\end{align}
	which by the construction of the local intervals (cf. Lemma \ref{lemma.localization}) is approximately the sum of the $d$-th moments on $\widetilde{I}_m$, while with a soft transition at the boundary. The main advantage of the smoothed moments $M_{m,d}$ is that it admits the following unbiased estimator: 
	\begin{align*}
	\bE\left[\sum_{j=1}^k \sum_{s\in nI_m/2} \bP\left(\mathsf{B}\left(h_j, \frac{1}{2}\right) = s \right)\cdot g_{d,x_m}\left(\frac{h_j - s}{n/2}\right)  \right] = M_{m,d}. 
	\end{align*}
	The above estimator motivates the current form of the smoothed moment estimators $\widehat{M}_{m,d}$ in \eqref{eq:est_smoothed_moments}, and makes use of an implicit sample splitting $\bE[\1(\widehat{p}_{j,1}\in I_m)g_{d,x_m}(\widehat{p}_{j,2}) \mid h_j ]$. 
	\item The final linear program in \eqref{eq:linear_program}: as has been discussed in the high-level idea, the specific form of the linear program in \eqref{eq:linear_program} is guided by two principles. First, the objective function $L$ upper bounds the target Wasserstein distance, thus the minimization of $L$ also essentially minimizes the Wasserstein distance. Second, the true associated measure $\mu_p$ is always a feasible solution to \eqref{eq:linear_program}, and therefore the triangle inequality can be applied to upper bound the distance between the estimator $\widehat{\mu}$ and the truth $\mu_p$. Regarding the specific form of the objective $L$, we aim to find a probability measure $\widehat{\mu}_m$ to match the smoothed local moment estimators $\widehat{M}_{m,d}$ up to order $D$ in each local interval, with appropriate weights for each local interval and each degree. We also treat the case $d=0$ differently, for $\sum_{m'\ge m} M_{m',0}$ is the right quantity to look at on bounding the Wasserstein distance, and we will show that the estimation error of the above sum is roughly the same as that of the single term $M_{m,0}$. For the constraints, the inequality $\widehat{\mu}_0(\bR)\le 1$ ensures that $\widehat{\mu}_0$ can be modified to a probability measure by adding mass on zeros, and the mean constraint $\int_0^1 x\widehat{\mu}_0(dx)\le 1/k$ is mostly for technical purposes. 
	\item Computational complexity: the main computation lies in the computation of $\widehat{M}_{m,d}$'s and solving the linear program \eqref{eq:linear_program}. Based on \eqref{eq:g_function}--\eqref{eq:est_smoothed_moments}, the evaluation of each function $g_{d,x_m}$ takes $\widetilde{O}(1)$ time, and thus the computation complexity of each $\widehat{M}_{m,d}$ is at most $\widetilde{O}(k\cdot n\widetilde{\ell}_m)$. Therefore the overall time complexity for the computation of moment estimators is $\widetilde{O}(nk)$. As for the linear program \eqref{eq:linear_program}, by a quantization of the support of $\widehat{\mu}_0$ it can also be solved in polynomial time. Finally, we remark that for the result in Theorem \ref{thm:PML_main} to hold, we do not need our estimator $\widehat{\mu}$ to be efficiently computable. 
\end{itemize}

The performance of the estimator $\widehat{\mu}$ is summarized in the following theorem. 
\begin{theorem}\label{thm:LMM}
Let $c_1>0$ be a large enough constant given in Lemma \ref{lemma:local_moments_expectation} and \ref{lemma.localization}, $c_2 \le c_0\delta$ for some small constant $c_0>0$ depending only on $c_1$, and $c_2\log n\ge 1$. Then under the Poissonized model, the above estimator $\widehat{\mu}$ satisfies 
\begin{align*}
\bP\left( k\cdot \text{\rm W}_1(\widehat{\mu}, \mu_p) \ge c\sqrt{\frac{k}{n\log n}} + \varepsilon \right) \le \exp(-c'n^{1-\delta}\varepsilon^2)
\end{align*}
for all $p\in \calM_k$ and $\varepsilon \ge n^{\delta-1/3}$, where $c,c'>0$ are absolute constants. 
\end{theorem}
The proof of Theorem \ref{thm:LMM} will be placed in the next subsection. We show that although Theorem \ref{thm:LMM} assumes a Poissonized model, it also implies the claimed result in Theorem \ref{thm:sorted_main} under the original sampling model. The proof is via the following simple reduction: let $N\sim\mathsf{Poi}(n)$ be the random sample size used in the Poissonized model, and $E$ be the above failure event. Then
\begin{align*}
\bP(E \mid N=n) = \frac{\bP(E, N=n)}{\bP(N=n)} \le \frac{\bP(E)}{\bP(N=n)} \le \sqrt{n}\cdot \exp(-c'n^{1-\delta}\varepsilon^2),
\end{align*}
where $\bP(E\mid N=n)$ is exactly the failure probability under the i.i.d. sampling model with fixed sample size $n$, and we have used that $\bP(\mathsf{Poi}(n)=n)\ge 1/\sqrt{n}$. Since the $\sqrt{n}$ factor does not offset the super-polynomially small deviation probability, the claimed concentration result in Theorem \ref{thm:sorted_main} holds. 

\subsection{Estimator Analysis}\label{subsec.analysis}
This section is devoted to the proof of Theorem \ref{thm:LMM}. We first upper bound the Wasserstein distance $\text{W}_1(\widehat{\mu}, \mu_p)$ in terms of the random quantity $L(\mu, \widehat{M})$, where $\mu\approx \mu_p$ is a carefully chosen feasible solution of the linear program \eqref{eq:linear_program}, and $L$ is the objective function in \eqref{eq:linear_program}. Subsequently, we upper bound the expectation $\bE[L(\mu,\widehat{M})]$ and the deivation probability $\bP(L(\mu,\widehat{M}) \ge \bE[L(\mu,\widehat{M})] + \varepsilon)$, respectively. 

\subsubsection{A deterministic upper bound of Wasserstein distance}
We first construct a candidate feasible solution $\mu$ of \eqref{eq:linear_program}. For $m\in [M]$, let 
\begin{align}\label{eq:local_measure}
\mu_m \triangleq \frac{1}{k}\sum_{j: p_j \in \widetilde{I}_m} \delta_{p_j}\cdot \bP\left(\mathsf{Poi}\left(\frac{np_j}{2}\right)\in \frac{n}{2}I_m\right)
\end{align}
be the target measure on the $m$-th local interval, and $\mu = \sum_{m=1}^M \mu_m$. Clearly $\text{supp}(\mu_m)\subseteq \widetilde{I}_m$, and 
\begin{align*}
\mu_p - \mu = \frac{1}{k}\sum_{j=1}^k \delta_{p_j} \cdot \sum_{m=1}^M \bP\left(\mathsf{Poi}\left(\frac{np_j}{2}\right) \in\frac{n}{2}I_m \right)\1(p_j\notin \widetilde{I}_m). 
\end{align*}
Hence, $\mu_p - \mu\ge 0$ is a non-negative measure, and Lemma \ref{lemma.localization} shows that $(\mu_p - \mu)(\bR)\le n^{-4}$. As a result, we have
\begin{align*}
\mu(\bR) \le \mu_p(\bR) = 1, \qquad \int_{[0,1]} x\mu(dx) \le \int_{[0,1]} x\mu_p(dx) = \frac{1}{k}\sum_{j=1}^k p_j = \frac{1}{k},
\end{align*}
and therefore $\mu$ is a feasible solution of \eqref{eq:linear_program}. 

Since $\widehat{\mu}(\bR) = \mu_p(\bR)=1$, we may set $f(0)=0$ in the dual representation of the Wasserstein-1 distance (cf. Lemma \ref{lemma.dual}). Therefore,
\begin{align}
\text{W}_1(\widehat{\mu}, \mu_p) &= \sup_{\|f\|_{\text{Lip}}\le 1, f(0)=0} \int_0^1 f(x)\left(\widehat{\mu}(dx) - \mu_p(dx) \right) \nonumber \\
&\le  \sup_{\|f\|_{\text{Lip}}\le 1, f(0)=0}  \sum_{m=1}^M\int_0^1 f(x)\left(\widehat{\mu}_m(dx) - \mu_m(dx) \right) + (\mu_p - \mu)(\bR) \nonumber\\
&\le \sup_{\|f\|_{\text{Lip}}\le 1, f(0)=0}  \sum_{m=1}^M\left[ \int_{\widetilde{I}_m} \left( f(x) - \sum_{d=0}^D a_{m,d}(x-x_m)^d \right) \left(\widehat{\mu}_m(dx) - \mu_m(dx) \right) \right. \nonumber \\
&\left. \qquad \sum_{d=0}^D a_{m,d}\left(\int_{\widetilde{I}_m}(x-x_m)^d\widehat{\mu}_m(dx) - \int_{\widetilde{I}_m}(x-x_m)^d\mu_m(dx) \right) \right] + n^{-4}, \label{eq:Wasserstein_bound}
\end{align}
where $a_{m,d}$ can be any real coefficients. We set these coefficients so that the degree-$D$ polynomial $P_m(x) = \sum_{d=0}^D a_{m,d}(x-x_m)^d$ satisfies the inequality \eqref{eq.approx_pointwise} on the interval $\widetilde{I}_m$ in Lemma \ref{lemma:jackson}. Then according to the pointwise inequality \eqref{eq.approx_pointwise}, for $m=1$ we have
\begin{align*}
\left| f(x) - \sum_{d=0}^D a_{1,d}(x-x_1)^d \right| \le \frac{C\sqrt{c_1}}{c_2}\cdot \sqrt{\frac{x}{n\log n}}, \qquad \forall x\in \widetilde{I}_1. 
\end{align*}
For $m\ge 2$, the norm bound \eqref{eq.approx_norm} leads to
\begin{align*}
\left| f(x) - \sum_{d=0}^D a_{m,d}(x-x_m)^d \right| \le \frac{C}{c_2}\cdot \frac{\widetilde{\ell}_m}{\log n}\le \frac{4C\sqrt{c_1}}{c_2}\cdot \sqrt{\frac{x}{n\log n}}, \qquad \forall x\in \widetilde{I}_m, 
\end{align*}
where the last inequality follows from $\widetilde{\ell}_m \le 4\sqrt{c_1x\log n/n}$ for all $x\in \widetilde{I}_m$ with $m\ge 2$. Combining the above cases, we conclude that for any $m\in [M]$ and $x\in \widetilde{I}_m$, 
\begin{align}\label{eq:approx_error}
 \left| f(x) - \sum_{d=0}^D a_{m,d}(x-x_m)^d \right| \le \frac{4C\sqrt{c_1}}{c_2}\cdot \sqrt{\frac{x}{n\log n}}. 
\end{align}
By choosing $x=x_m$ in \eqref{eq:approx_error}, we conclude that the same result (possibly with coefficients doubled) still holds if we set $a_{m,0} = f(x_m)$ for each $m\in [M]$.\footnote{The first interval requires some special attention since $x\in \widetilde{I}_1$ can be as small as zero. However, since $x_1= 0$, the pointwise inequality \eqref{eq.approx_pointwise} already implies that $a_{1,0}=0$, and setting $a_{1,0}=f(x_1)$ makes the coefficient unchanged.}

Next we derive upper bounds on the magnitude of the coefficients $a_{m,d}$. Note that \eqref{eq:approx_error} implies that for any $m\in [M]$ and $x\in \widetilde{I}_m$, 
\begin{align*}
 \left| f(x_m) - \sum_{d=0}^D a_{m,d}(x-x_m)^d \right| &\le |f(x) - f(x_m)| +  \left| f(x) - \sum_{d=0}^D a_{m,d}(x-x_m)^d \right| \\
 &\le \widetilde{\ell}_m + \frac{4C}{c_2}\cdot \widetilde{\ell}_m = \left(1 + \frac{4C}{c_2}\right)\widetilde{\ell}_m. 
\end{align*}
Therefore, Lemma \ref{lemma:coefficients} applied to the shifted interval $\widetilde{I}_m - x_m$ gives that for $d\ge 1$, 
\begin{align}\label{eq:coefficient_bound}
|a_{m,d}| \le \left(1 + \frac{4C}{c_2}\right)\widetilde{\ell}_m^{1-d}\cdot \begin{cases}
2^{9D/2+1}& \text{if }m=1 \\
(1+\sqrt{2})^D & \text{if }m\ge 2
\end{cases} \le 2\left(1+\frac{4C}{c_2}\right) \widetilde{\ell}_m^{1-d}\cdot n^{9c_2/2}. 
\end{align}
Moreover, for $d=0$, the $1$-Lipschitz assumption of $f$ gives
\begin{align}\label{eq:coefficient_difference}
|a_{m,0} - a_{m-1,0}| = |f(x_m) - f(x_{m-1})| \le |x_m - x_{m-1}| \le \widetilde{\ell}_m, \qquad \forall m\in [M]. 
\end{align}

Combining \eqref{eq:Wasserstein_bound}--\eqref{eq:coefficient_difference} and using the Abel transformation $$\sum_{m=1}^M a_m z_m = \sum_{m=1}^M (a_m - a_{m-1})\sum_{m'\ge m} z_{m'}$$ with $a_0\triangleq 0$, we conclude that
\begin{align*}
\text{W}_1(\widehat{\mu}, \mu_p) &\le C'\sum_{m=1}^M \left[\int_{\widetilde{I}_m} \sqrt{\frac{x}{n\log n}}(\widehat{\mu}_m(dx) + \mu_m(dx)) \right. \\
& \qquad + n^{9c_2/2} \sum_{d=1}^D \widetilde{\ell}_m^{1-d}\left|\int_{\widetilde{I}_m}(x-x_m)^d\widehat{\mu}_m(dx) - \int_{\widetilde{I}_m}(x-x_m)^d\mu_m(dx) \right| \\
&\left. \qquad + \widetilde{\ell}_m \left|\sum_{m'\ge m} \widehat{\mu}_{m'}(\widetilde{I}_{m'}) - \sum_{m'\ge m} \mu_{m'}(\widetilde{I}_{m'})  \right| \right] + n^{-4} \\
&\le C' \int_0^1  \sqrt{\frac{x}{n\log n}}(\widehat{\mu}(dx) + \mu(dx))  + C'n^{9c_2/2}k^{-1}(L(\widehat{\mu}_0, \widehat{M}) + L(\mu, \widehat{M})) + n^{-4},
\end{align*}
where the last step follows from the definition of $L$ in \eqref{eq:linear_program} and the triangle inequality. Since $\mu$ is a feasible solution of \eqref{eq:linear_program}, we have $L(\widehat{\mu}_0, \widehat{M})\le L(\mu, \widehat{M})$. Moreover, the feasibility condition of $\mu$ implies $\mu(\bR)\le 1$ and $\int_0^1 x\mu(dx)\le k^{-1}$, therefore the Cauchy--Schwartz inequality gives
\begin{align*}
\int_0^1 \sqrt{x} \mu(dx) \le \sqrt{\mu(\bR)\cdot \int_0^1 x\mu(dx)} \le \frac{1}{\sqrt{k}}. 
\end{align*}
The same result also holds for $\widehat{\mu}_0$. 

A combination of the above inequalities gives the deterministic upper bound of the Wasserstein-$1$ distance: 
\begin{align}\label{eq:Wasserstein_final}
k\cdot \text{W}_1(\widehat{\mu}, \mu_p) \le 2C'\left(\sqrt{\frac{k}{n\log n}} + n^{9c_2/2}\cdot L(\mu, \widehat{M}) \right) + \frac{k}{n^4},
\end{align}
where the constant $C'>0$ depends only on $(c_1,c_2)$. 

\subsubsection{Bounding the expectation of $L(\mu, \widehat{M})$}
This section is devoted to the upper bound of the expected loss $\bE[L(\mu, \widehat{M})]$, where it suffices to show that the smoothed moment estimators $\widehat{M}_{m,d}$ in \eqref{eq:est_smoothed_moments} are close to the true smoothed moments $M_{m,d}$ in \eqref{eq:smoothed_moments}. This is the main subject of the following lemma. 
\begin{lemma}\label{lemma:local_moments_expectation}
Consider the estimator $\widehat{M}_{m,d}$ in \eqref{eq:est_smoothed_moments} and the measure $\mu_m$ in \eqref{eq:local_measure}. For $c_1>0$ large enough, there exist constants $C,C'>0$ depending only on $c_1$ such that for $1\le d\le D$, 
\begin{align*}
\bE\left|k\cdot \int_{\widetilde{I}_m} (x-x_m)^d \mu_m(dx) - \widehat{M}_{m,d} \right| \le Cn^{C'c_2}\widetilde{\ell}_m^d\left(\sqrt{k_m} + \frac{k}{n^5}\right),
\end{align*}
and for $d=0$, 
\begin{align*}
\bE\left|\sum_{m'\ge m} \left(k\cdot \mu_{m'}(\widetilde{I}_{m'}) - \widehat{M}_{m',0}\right) \right| \le C\left(\sqrt{k_m\log n} + \log n + \frac{k}{n^5}\right),
\end{align*}
where $k_m$ is the ``effective support size" of the restriction of $p$ on $I_m$: 
\begin{align}\label{eq.effective_sample_size}
k_m \triangleq \sum_{j=1}^k \bP\left(\mathsf{Poi}\left(\frac{np_j}{2}\right) \in \frac{n}{2}I_m\right). 
\end{align}
\end{lemma}

As a direct consequence of Lemma \ref{lemma:local_moments_expectation}, the expected loss is upper bounded by
\begin{align*}
\bE[L(\mu, \widehat{M})] &\le Cn^{C'c_2}\sum_{m=1}^M \widetilde{\ell}_m \cdot (D+1)\left(\sqrt{k_m\log n} + \log n + \frac{k}{n^5}\right) \\
&\le C(c_2\log n + 1)n^{C'c_2}\log n\cdot \sum_{m=1}^M \widetilde{\ell}_m \left(\sqrt{k_m} + 1 + \frac{k}{n^5}\right).
\end{align*}
To upper bound the above quantity, we shall need the following lemma. 
\begin{lemma}\label{lemma:total_variance}
There exists an absolute constant $C>0$ depending only on $c_1$ such that
\begin{align}
\sum_{m=1}^M \widetilde{\ell}_m \sqrt{k_m} \le C\left( \left(\frac{\log n}{n}\right)^{\frac{1}{3}} + \sqrt{\frac{k}{n^5}}\right). 
\end{align}
\end{lemma}

By Lemma \ref{lemma:total_variance} and the inequality $\sum_{m=1}^M \widetilde{\ell}_m\le 2$, we arrive at the final upper bound: 
\begin{align}\label{eq:expectation_final}
\bE[L(\mu, \widehat{M})] = O\left(n^{-1/3 + O(c_2)}\log n + \sqrt{\frac{k}{n^5}} + \frac{k}{n^5} \right). 
\end{align}

\subsubsection{Bounding the deviation probability of $L(\mu, \widehat{M})$}
In this section we apply the McDiarmid's inequality (cf. Lemma \ref{lemma:mcdiarmid}) for the deviation probability of $L(\mu, \widehat{M})$, where a key step is to show that each moment estimator $\widehat{M}_{m,d}$ does not change much by altering one observation. 

\begin{lemma}\label{lemma:bounded_diff}
For $m\in [M]$, $0\le d\le D$, let $\Delta_{m,d}(h)$ be the difference between the estimators $\widehat{M}_{m,d}$ given the histograms $h=(h_1,\cdots,h_k)$ and $h'=(h_1,\cdots,h_{j-1},h_{j}+1,h_{j+1},\cdots,h_k)$, respectively. Let $\calM\subseteq [M]$ be the set of indices of ``active" local intervals, i.e.
\begin{align*}
\calM = \{m\in [M]: nx_{m,\text{\rm L}} \le h_j \le n x_{m,\text{\rm R}} \}. 
\end{align*}
Then there exists an absolute constant $C>0$ depending on $(c_1, c_2)$ such that for $1\le d\le D$, 
\begin{align*}
|\Delta_{m,d}| \le C\cdot \begin{cases}
\frac{\log n}{n}\cdot (x_{m,\text{\rm R}} - x_{m,\text{\rm L}})^{d-1} & \text{if } m\in \calM, \\
n^{-5}\cdot (x_{m,\text{\rm R}} - x_{m,\text{\rm L}})^{d} & \text{if } m\notin \calM. 
\end{cases}
\end{align*}
Also, for $d=0$, it holds that
\begin{align*}
\left|\sum_{m'\ge m} \Delta_{m,0} \right| \le C\cdot \begin{cases}
\log n/(n(x_{m,\text{\rm R}} - x_{m,\text{\rm L}})) & \text{if } m\in \calM, \\
n^{-5}& \text{if } m\notin \calM. 
\end{cases}
\end{align*}
\end{lemma}

Based on Lemma \ref{lemma:bounded_diff}, we are ready to upper bound the difference in $L(\mu, \widehat{M})$ when we change the histogram from $h$ to $h'$ as in Lemma \ref{lemma:bounded_diff}. Note that $|\calM|\le 2$ always holds, and $\widetilde{\ell}_m\le x_{m,\text{R}} - x_{m,\text{L}} \le 2\widetilde{\ell}_m$, we conclude that
\begin{align*}
|L(\mu,\widehat{M}(h)) - L(\mu,\widehat{M}(h')) | &\le \sum_{m=1}^M \widetilde{\ell}_m\left(\sum_{d=1}^D \widetilde{\ell}_m^{-d}|\Delta_{m,d}| + \left|\sum_{m'\ge m} \Delta_{m,0} \right|\right) \\
&\le C \sum_{m\in \calM}\sum_{d=0}^D 2^d \cdot \frac{\log n}{n} + C\sum_{m\notin \calM} \sum_{d=0}^D 2^d \cdot \frac{\widetilde{\ell}_m}{n^5} \\
&\le C\left(\frac{2^{D+2}\log n}{n} + \frac{2^{D+2}}{n^5}\right) \\
&= O(n^{c_2-1}\log n).
\end{align*}
Hence, treating $L(\mu, \widehat{M})$ as a deterministic function of the observations $X_1, \cdots, X_N$ under the Poissonized model, it is clear that this function is invariant with the permutation of the observations (for $\widehat{M}$ only depends on the histrogram or even the profile), and the bounded difference condition of Lemma \ref{lemma:mcdiarmid} is satisfied with $c = O(n^{c_2-1}\log n)$. Consequently, the McDiarmid's inequality under Poisson sampling (cf. Lemma \ref{lemma:mcdiarmid}) leads to
\begin{align}\label{eq:concentration_final}
\bP\left(L(\mu,\widehat{M}) \ge \bE[L(\mu, \widehat{M})] + \varepsilon \right) \le 4\exp\left(-\frac{cn\varepsilon^2}{(n^{c_2}\log n)^2}\right)
\end{align}
for some absolute constant $c>0$. 

Finally, a combination of \eqref{eq:Wasserstein_final}, \eqref{eq:expectation_final} and \eqref{eq:concentration_final} with a small enough $c_2>0$ completes the proof of Theorem \ref{thm:LMM} (since $k\cdot \text{W}_1(\widehat{\mu}, \mu_p)\le 2$ trivially holds, it may be assumed that $k=O(n\log n)$ in the statement of Theorem \ref{thm:LMM}).

\subsection{Implications on Poisson Approximation}\label{subsec.poisson_approx}
The estimator construction in Section \ref{subsec.estimator} also sheds light on other functional estimation problems and approximation theory. In the previous works on functional estimation, usually the parameter space is split into several regions where different estimation procedures are performed. However, since one does not have the perfect knowledge of the unknown parameter, sample splitting is typically employed to separate the tasks of region classification and statistical estimation. For example, under the discrete distribution models, the following type of estimator is common: 
\begin{align}\label{eq:functional_sample_splitting}
T(\widehat{p}) = T_1(\widehat{p}_1)\1(\widehat{p}_2\in R_1) + T_2(\widehat{p}_1)\1(\widehat{p}_2\in R_2),
\end{align}
where $\widehat{p}_1, \widehat{p}_2$ are the splitted observations of the empirical probability $\widehat{p}$, and $T_1, T_2$ are the estimators in the regions $R_1$ and $R_2$, respectively. However, the estimator \eqref{eq:functional_sample_splitting} may be unstable to small changes of input and incur a large variance without proper truncation. Our estimator construction in Section \ref{subsec.estimator} provides a simple fix to this problem, where we may reduce the variance by using the following modified estimator
\begin{align}
T'(\widehat{p}) &= \bE\left[T_1(\widehat{p}_1)\1(\widehat{p}_2\in R_1) + T_2(\widehat{p}_1)\1(\widehat{p}_2\in R_2) \mid \widehat{p} \right] \nonumber\\
&=\sum_{j\in R_1}\bP\left(\mathsf{B}\left(n\widehat{p}, \frac{1}{2}\right) = j \right)T_1\left(\frac{n\widehat{p}-j}{n/2}\right) + \sum_{j\in R_2}\bP\left(\mathsf{B}\left(n\widehat{p}, \frac{1}{2}\right) = j \right)T_2\left(\frac{n\widehat{p}-j}{n/2}\right). \label{eq:functional_no_sample_split}
\end{align}
Note that the deviation probability results (cf. Lemma \ref{lemma:bounded_diff}) highly rely on the modified form \eqref{eq:functional_no_sample_split}. 

A striking result is that \eqref{eq:functional_no_sample_split} also helps in approximation theory, a seemingly unrelated field to statistical estimation. A crux of the analysis of the population maximum likelihood estimator in \cite{vinayak2019maximum,vinayak2019optimal} is the following approximation-theoretic result: for any $1$-Lipschitz function $f$ on $[0,1]$ with $f(0) = 0$, there exists a degree-$n$ Bernstein polynomial $$p_n(x) = \sum_{k=0}^n a_{n,k} \cdot \binom{n}{k}x^k(1-x)^{n-k}$$
with a good approximation property $\|f - p_n\|_\infty = O(1/\sqrt{n\log n})$ as well as a small magnitude of coefficients $\max_k |a_{n,k}| = O(n^{1.5 + \varepsilon})$ for some small $\varepsilon>0$. In their applications, small coefficients are desirable to obtain a small variance for the maximum likelihood estimator, and their proof of the above fact relies on some complicated degree raising process of Bernstein polynomials and their relationships to the Chebyshev polynomials. In contrast, we will explicitly relate the above problem to the statistical estimation, and employ the idea of \eqref{eq:functional_no_sample_split} to improve the above coefficient bound from $O(n^{1.5 + \varepsilon})$ to $O(1)$. 

The next theorem presents a better coefficient bound for the Poisson approximation polynomials, which are slight variants of the Bernstein polynomials. Adapting the proof slightly, the same result will also hold for Bernstein polynomials. 

\begin{theorem}\label{thm.poisson_approx}
	Let $f$ be any $1$-Lipschitz function on $\bR$ with $f(0)=0$. Then for any $\varepsilon, \delta>0$, there exists a constant $C>0$ depending only on $(\varepsilon,\delta)$ such that for any $n\in \naturals$, there is a sequence of coefficients $(b_j)_{j=0}^{\infty}$ with
	\begin{align}\label{eq.approx_error}
	\left|f(x) - \sum_{j=0}^{\infty} b_j\bP(\Poi(nx) = j) \right| \le C\sqrt{\frac{x}{n\log n}}, \quad \forall x\in [0,1],
	\end{align}
	where $b_j = 0$ for $j> (1+\delta)n$, and
	\begin{align}\label{eq.coeff_bound}
	\left|b_j - f\left(\frac{j}{n}\right) \right| \le \frac{C(1+j^{1/2})}{n^{1-\varepsilon}}, \quad \forall j\le (1+\delta)n.
	\end{align}
\end{theorem}

The following corollary is immediate by choosing any $\varepsilon \in (0,1/2)$. 
\begin{corollary}
For any $\delta>0$, there exists an absolute constant $C$ depending only on $\delta$ such that for any $1$-Lipschitz function $f$ on $\bR$ with $f(0)=0$, there exists a sequence of coefficients $(b_j)_{j=0}^{\infty}$ such that $b_j = 0$ for $j>(1+\delta)n$, $\max_j |b_j|\le C$ and $\max_{x\in [0,1]}|f(x)-\sum_{j=0}^\infty b_j\bP(\mathsf{Poi}(nx)=j)|\le C/\sqrt{n\log n}$.  
\end{corollary}

Note that a natural choice for the coefficient is $b_j = f(j/n)$, where one uses $\bE[f(X/n)]$ with $X\sim \Poi(nx)$ to approximate the 1-Lipschitz function $f(x)$. For this choice, we have the pointwise bound
$$
| \bE[f(X/n)] - f(x)| \le \bE |f(X/n) - f(x)| \le \bE|X/n - x| \le \sqrt{\bE(X/n-x)^2} = \sqrt{\frac{x}{n}}. 
$$
Hence, Theorem \ref{thm.poisson_approx} shows that we may improve the above pointwise bound uniformly by a logarithmic factor, while do not change the coefficients $b_j$ by too much. 

The recipe of the proof of Theorem \ref{thm.poisson_approx} is as follows. First, we partition the entire interval $[0,1]$ into local intervals as before, and for each local interval $I$, we construct local Poisson polynomials with the desired approximation property and coefficients essentially supported on $I$. Then we employ the idea of \eqref{eq:functional_no_sample_split} to piece together the previous local polynomials into a global Poisson polynomial with the desired properties. The detailed proof is deferred to the appendix. 
\section{Acknowledgement}
The authors are grateful to professors Jayadev Acharya, Moses Charikar, Aaron Sidford, and Tsachy Weissman for their tremendous support on this project, and their very helpful comments on the paper. We would also like to thank Jayadev Acharya once more for initializing this problem and pointing out the useful reference \cite{acharya2012tight}. Yanjun Han would like to thank Wei-Ning Chen for igniting discussions on Theorem \ref{thm.poisson_approx}. 
\bibliographystyle{alpha}
\bibliography{di}
\appendix

\section{Auxiliary Lemmas}
\begin{lemma}[Local intervals, Lemma 17 of \cite{han2018local}]\label{lemma.localization}
Let $I_m$ and $\widetilde{I}_m$ be defined in \eqref{eq:local_interval} and \eqref{eq:local_interval_enlarged}, respectively, with $m\in [M]$. Then for $c_1>0$ large enough, the following inequality holds for each $p\in [0,1]$, and $h\sim \mathsf{Poi}(np)$ or $h\sim \mathsf{B}(n,p)$: 
\begin{align*}
\bP(h\notin n\widetilde{I}_m \mid p\in I_m) &\le n^{-5}, \\
\bP(h\in nI_m \mid p\notin \widetilde{I}_m) & \le n^{-5}. 
\end{align*}
\end{lemma}

\begin{lemma}[Dual representation of Wasserstein distance \cite{kantorovich1958space}]\label{lemma.dual}
For two Borel probability measures $P, Q$ on a separable metric space $(S, d)$, the following duality result holds:
\begin{align*}
\text{\rm W}_1(P,Q) = \sup_{f: \|f\|_{\text{\rm Lip}}\le 1} \bE_P[f(X)] - \bE_Q[f(X)], 
\end{align*}
where $X$ is a random variable taking value in $S$ with distribution $P$ or $Q$, and the Lipschitz norm is defined as $ \|f\|_{\text{\rm Lip}} = \sup_{x\neq y\in S} |f(x) - f(y)|/d(x,y)$. 
\end{lemma}

\begin{lemma}[Jackson's theorem \cite{devore1976degree}]\label{lemma:jackson}
Let $D>0$ be any integer, and $[a,b]\subseteq \mathbb{R}$ be any bounded interval. For any $1$-Lipschitz function $f$ on $[a,b]$, there exists a universal constant $C$ independent of $D,f$ such that there exists a polynomial $P(\cdot)$ of degree at most $D$ such that
\begin{align}\label{eq.approx_pointwise}
|f(x) - P(x)| \le \frac{C\sqrt{(b-a)(x-a)}}{D}, \qquad \forall x\in [a,b].
\end{align}
In particular, the following norm bound holds:
\begin{align}\label{eq.approx_norm}
\sup_{x\in [a,b]} |f(x)-P(x)| \le \frac{C(b-a)}{D}.
\end{align}
\end{lemma}

\begin{lemma}[Coefficient bound, Lemma 28 of \cite{han2016minimax}]\label{lemma:coefficients}
	Let $p_n(x) = \sum_{\nu=0}^n a_\nu x^\nu$ be a polynomial of degree at most $n$ such that $|p_n(x)|\leq A$ for $x\in [a,b]$. Then
	\begin{enumerate}
		\item If $a+b\neq 0$, then
		\begin{align*}
		|a_\nu| \le 2^{7n/2}A\left|\frac{a+b}{2}\right|^{-\nu}\left(\left|\frac{b+a}{b-a}\right|^n +1 \right), \qquad \nu=0,\cdots,n.
		\end{align*}
		\item If $a+b = 0$, then
		\begin{align*}
		|a_\nu| \leq A b^{-\nu} (\sqrt{2}+1)^n, \qquad \nu=0,\cdots,n.
		\end{align*}
	\end{enumerate}
\end{lemma}

\begin{lemma}[Poisson tail inequality, Theorem 5.4 of \cite{mitzenmacher2005probability}]\label{lemma:poissontail}
	For $X\sim \mathsf{Poi}(\lambda)$ or $X\sim \mathsf{B}(n, \lambda/n)$ and any $\delta>0$, we have
	\begin{align*}
	\mathbb{P}(X\ge (1+\delta)\lambda) &\le \left(\frac{e^\delta}{(1+\delta)^{1+\delta}}\right)^\lambda \le \exp\left(-\frac{(\delta^2\wedge \delta)\lambda}{3}\right),\\
	\mathbb{P}(X\le (1-\delta)\lambda) &\le \left(\frac{e^{-\delta}}{(1-\delta)^{1-\delta}}\right)^\lambda \le \exp\left(-\frac{\delta^2\lambda}{2}\right).
	\end{align*}
\end{lemma}

\begin{lemma}[Bennett's inequality \cite{bennett1962probability}]\label{lemma:bennett}
	Let $X_1,\ldots,X_n\in [a,b]$ be independent random variables with 
	\begin{align*}
	\sigma^2 \triangleq \sum_{i=1}^n \var(X_i).
	\end{align*}
	Then we have
	\begin{align*}
	\bP\left(\left|\sum_{i=1}^n X_i - \sum_{i=1}^n \bE[X_i]\right|\ge \varepsilon\right)\le 2\exp\left(-\frac{\varepsilon^2}{2(\sigma^2+(b-a)\varepsilon/3)}\right).
	\end{align*}
\end{lemma}

\begin{lemma}[McDiarmid's inequality under Poisson sampling, Lemma 12 of \cite{hao2019unified}]\label{lemma:mcdiarmid}
Let $N\sim \mathsf{Poi}(n)$, and $X_1,X_2, \cdots$ be independent random variables taking value in $\calX$ independent of $N$. Assume that the following bounded difference condition holds for a sequence of real-valued functions $f_n: \calX^n \to \bR$: 
\begin{align*}
\sup_n \sup_{i\in [n]} \sup_{x_1,x_2,\cdots,x_n, x_i'\in \calX} |f_n(x_1,\cdots,x_n) - f_{n+1}(x_1,\cdots,x_{i-1},x_i,x_i',x_{i+1},\cdots,x_n)| \le c.
\end{align*}
Then for any $\varepsilon>0$, 
\begin{align*}
\bP\left(f_N(X_1,\cdots,X_N) \ge \bE[f_N(X_1,\cdots,X_N)] + \varepsilon \right) \le 4\exp\left(-\frac{n\varepsilon^2}{32(n^2c^2 + 1)}\right). 
\end{align*}
\end{lemma}

\begin{lemma}[Some properties of Charlier polynomials, Proof of Lemma 13 of \cite{hao2019unified}, and Lemma 30 of \cite{han2018local}]\label{lemma:charlier}
Let the function $g_{d,x}(\cdot)$ be defined in \eqref{eq:g_function}, with $d\in \NN$, $x\in [0,1]$. Then for any $z\in [0,1]$ and $d\ge 1$, the following identity holds: 
\begin{align*}
g_{d,x}\left(z+\frac{2}{n}\right) - g_{d,x}(z) = \frac{2d}{n}g_{d-1,x}(z). 
\end{align*}
Moreover, if $nz/2\in \NN$ and $\max\{|z-x|, 8d/n, \sqrt{8zd/n} \}\le \Delta$, then
	\begin{align*}
	|g_{d,x}(z)|  \le (2\Delta)^d.
	\end{align*}
\end{lemma}

\begin{lemma}[Difference of Binomial CDF, Corollary 4 of \cite{jiao2020bias}]\label{lemma:cdf_diff}
For $n\in \NN$, let $F_n$ be the CDF of a normalized Binomial random variable $X/n$ with $X\sim \mathsf{B}(n,1/2)$, i.e. $F_n(t) = \bP(X/n \le t)$. Then there exists an absolute constant $C>0$ such that for any $t\in [0,1]$ and $n\ge 1$, 
\begin{align*}
|F_{n+1}(t) - F_n(t)| \le \frac{C}{\sqrt{n}}. 
\end{align*}
\end{lemma}

\section{Proof of Main Lemmas}
\subsection{Proof of Corollary \ref{cor:functional}}
If $\varepsilon'=0$, then following the same lines as \cite{acharya2017unified} shows that the PML plug-in estimator attains the rate-optimal sample complexity for estimating these properties to accuracy $2\varepsilon$. Hence, here it suffices to prove that the additional error $\varepsilon'$ satisfies $\varepsilon' = O(n^{-1/3}) = o(\varepsilon)$ for a large enough $\constA>0$ in all these examples. 
\begin{enumerate}
	\item Entropy: to obtain a meaningful estimation performance for entropy, we need $k=O(n\log n)$ \cite{Valiant--Valiant2011,Jiao--Venkat--Han--Weissman2015minimax,wu2016minimax}. Therefore, for each $q\in\calM_k$ which is $(n^{-\constA},3n^{-\constA/2})$-close to the distribution $p$, we have
	\begin{align*}
	d(p,q) &\le \sum_{i=1}^k \left|p_i\log p_i - q_i\log q_i \right| \\
	&\le \sum_{i: p_i\le n^{-\constA}} \left(|p_i\log p_i| + |q_i\log q_i| \right) + \sum_{i: p_i>n^{-\constA}} |p_i\log p_i - q_i\log q_i|. 
	\end{align*}
	For any $p_i>n^{-\constA}$, we have
	\begin{align*}
	\frac{p_i}{1+3n^{-\constA/2}}\log\frac{1}{p_i} \le q_i\log \frac{1}{q_i} \le p_i\log \frac{1+3n^{-\constA/2}}{p_i}.
	\end{align*}
	Therefore, 
	\begin{align*}
	d(p,q) &\le 2k\cdot \frac{\constA\log n}{n^{\constA}} + \sum_{i: p_i>n^{-\constA}} \left(\frac{3n^{-\constA/2}}{1+3n^{-\constA/2}}\cdot p_i\log \frac{1}{p_i} + p_i\log(1+3n^{-\constA/2})\right) \\
	&= O\left(\frac{k}{n^{\constA}} + \frac{\log k}{n^{\constA/2}} \right),
	\end{align*}
	and choosing any $\constA \ge 2$ gives that $\varepsilon' = O(n^{-1/3})$. 
	\item Support size: it is assumed in \cite{wu2019chebyshev} that $p_i\ge 1/k$ for all $i$ and $k=O(n\log n)$. Therefore, after choosing $\constA \ge 2$, any distribution $q$ which is $(n^{-\constA},3n^{-\constA/2})$-close to the distribution $p$ satisfying the above assumption must have the same support as $p$. Consequently, $d(p,q) = 0$, and $\varepsilon'=0$ for the support size estimation. 
	\item Support coverage: by \cite{orlitsky2016optimal}, for the support coverage estimation it amounts to consider the functional $S_m(p) = \sum_{i=1}^k (1-(1-p_i)^m)$ with $m=O(n\log n)$. For any distribution $q$ which is $(n^{-\constA},3n^{-\constA/2})$-close to the distribution $p$, we have
	\begin{align*}
	d(p,q) &= |S_m(p) - S_m(q)| = \left|\sum_{i=1}^k \left(1 - mp_i - (1-p_i)^m \right) - \sum_{i=1}^k \left(1 - mq_i - (1-q_i)^m \right) \right| \\
	&\le \sum_{i: p_i\le n^{-\constA}} \left( \left|1-mp_i - (1-p_i)^m \right| + \left|1 - mq_i - (1-q_i)^m \right|\right) \\
	&\qquad + \sum_{i: p_i>n^{-\constA}} \left|m(p_i-q_i) + (1-p_i)^m - (1-q_i)^m \right| \\
	&\stepa{\le} \sum_{i: p_i\le n^{-\constA}} m^2(p_i^2 + q_i^2) + \sum_{i: p_i>n^{-\constA}} 2m \cdot |p_i - q_i| \\
	&\le \frac{m^2}{n^{\constA}}\sum_{i=1}^k (p_i + q_i) + \sum_{i=1}^k 2m\cdot \frac{p_i}{1+3n^{-\constA/2}} \\
	&= O\left(\frac{m^2}{n^{\constA}} + \frac{m}{n^{\constA/2}}\right), 
	\end{align*}
	where (a) follows from the mean value theorem applied to the function $p\mapsto mp + (1-p)^m$. Hence, choosing $\constA \ge 3$ gives $\varepsilon'=O(n^{-1/3})$.
	\item Distance to uniformity: by \cite{Valiant--Valiant2011power,jiao2018minimax}, we may assume that $k=O(n\log n)$. Then it is easy to see that for close pairs $(p,q)$, we have
	\begin{align*}
	d(p,q) = \left|\sum_{i=1}^k \left|p_i - \frac{1}{k}\right| - \sum_{i=1}^k \left|q_i - \frac{1}{k}\right|\right| \le \sum_{i=1}^k |p_i - q_i| = O\left(\frac{k}{n^{\constA}} + \frac{1}{n^{\constA/2}}\right).
	\end{align*}
	Hence, choosing any $\constA \ge 2$ suffices to give $\varepsilon'=O(n^{-1/3})$.
\end{enumerate}

\subsection{Proof of Lemma \ref{lemma:continuity}}
Define a grid $\calC=\{c_0,c_1,\cdots,c_M\}$ on $[0,1]$ as follows: let $c_0=0$, and for $i=1,\cdots,M$, let
\begin{align*}
c_i = \frac{1}{2n^{\constA}}\left(1 + \frac{1}{n^r}\right)^{i-1}, 
\end{align*}
where $M = \Theta(n^r\log n)$ is the solution to $c_M\le 1<c_{M+1}$. Define $\calQ$ to be the set of all discrete measures $q$ supported on $\NN$ with all masses belonging to the grid $\calC$ and total mass $q(\NN)$ at most one. Clearly, the number of symbols with probability mass $c_i$ under any $q\in\calQ$ is at most $2n^{\constA}$ for each $i=1,\cdots,M$. Hence, modulo the equivalence relation that $q\sim q'$ if $q$ coincides with $q'$ after some permutation of the domain elements and adding/removing domain elements with zero probability, we have
\begin{align*}
|\calQ| \le (2n^{\constA})^M = \exp\left(\Theta(n^r\log n)\right). 
\end{align*}
Now let $\calN$ be the set of all normalized probability measures $q/q(\NN)$ with $q\in \calQ$, clearly $|\calN| \le |\calQ| = \exp(\Theta(n^r\log n))$. Moreover, since $q(\NN)\le 1$ and $q$ only takes value in $\calC$, each probability measure in $\calN$ also has a minimum non-zero probability mass at least $1/(2n^{\constA})$, which implies $\calN\subseteq \calM_0$ as claimed.

Next we prove that the above set $\calN$ also satisfies the second condition of Lemma \ref{lemma:continuity}. We first introduce the following $\chi^m$-divergence for $m\ge 2$: 
\begin{align*}
\chi^m(P\|Q) = \bE_Q\left[\left(\frac{dP}{dQ}\right)^m \right].
\end{align*}
It is easy to verify the product rule $\chi^m(P_1\otimes P_2\|Q_1\otimes Q_2) = \chi^m(P_1\|Q_1)\cdot \chi^m(P_2\|Q_2)$. The data processing inequality for $\chi^m$ also holds because of the convexity of function $t\mapsto t^m$ on $\bR_+$. The following inequality will be very useful: for $P=\Poi(\lambda_1), Q=\Poi(\lambda_2)$ with $|\lambda_1/\lambda_2 - 1|\le \delta < 1/m$, it holds that
\begin{align}
\chi^m(P\|Q) &= \sum_{t=0}^\infty e^{(m-1)\lambda_2 - m\lambda_1}\cdot \frac{1}{t!}\left(\frac{\lambda_1^m}{\lambda_2^{m-1}}\right)^t \nonumber = \exp\left(\lambda_2 \left( \left(\frac{\lambda_1}{\lambda_2}\right)^m - m\left(\frac{\lambda_1}{\lambda_2} - 1\right) - 1 \right) \right) \nonumber\\
&\le \exp\left(\lambda_2\left(e^{m(\lambda_1/\lambda_2-1)} - m\left(\frac{\lambda_1}{\lambda_2} - 1\right) - 1 \right) \right) \le \exp(\lambda_2 m^2\delta^2), \label{eq:chi_m_divergence}
\end{align}
where the last step follows from $e^x\le 1+x+x^2$ for $x\le 1$. 

Now we consider the following Poissonized model: for each non-negative measure $q$ supported on $\NN$ (not necessarily a probability measure), let the distribution $q_h$ of histograms under $q$ be the product distribution $\prod_{j\in\NN} \Poi(nq_j)$. An equivalent definition of $q_h$ is the distribution of histograms when one draws $N$ i.i.d. samples from the normalized probability distribution $q/q(\NN)$, with an independent Poisson distributed sample size $N\sim \Poi(nq(\NN))$. We will write ${\bf q}_h$ to be the distribution over the space of profiles $\cup_{t=1}^\infty \Phi_t$ induced by $q_h$. 

We claim that for each distribution $p\in \calM_0$, there exists some $q\in \calQ$ such that both quantities $\chi^m({\bf p}_h \|{\bf q}_h)$ and $\chi^m({\bf q}_h \|{\bf p}_h)$ are small. Specifically, we define $q = (q_1,q_2,\cdots)$ such that
\begin{align}\label{eq:lower_quantization}
q_j = \max\{c_i\in \calC: c_i\le p_j \}. 
\end{align}
Since $p\in \calM_0$, we have $q\in \calQ$. Moreover, $\max\{|p_j/q_j - 1|, |q_j/p_j -1| \}\le n^{-r}$ for each $j\in \NN$. Since the profile is a deterministic function of the histogram, by the data processing inequality of the $\chi^m$-divergence, as long as $m\le n^{r}$ we have\footnote{Note that the profile distribution ${\bf p}_h$ is invariant with permutations of $p$ and adding/removing zero-probability symbols in $p$, the quotient with the equivalence relation used in the upper bound of $|\calQ|$ may be taken.}
\begin{align}\label{eq:chi_m_profile}
\chi^m({\bf p}_h\|{\bf q}_h) \le \chi^m(p_h \| q_h) \le \prod_{j\in\NN} \exp\left(nq_jm^2n^{-2r} \right) \le \exp\left(n^{1-2r}m^2\right), 
\end{align}
where the second inequality follows from the product rule of the $\chi^m$-divergence and \eqref{eq:chi_m_divergence}, and the last inequality follows from $\sum_{j\in\NN} q_j\le \sum_{j\in\NN} p_j = 1$. Analogously, we also have $\chi^m({\bf q}_h\|{\bf p}_h) \le \exp\left(n^{1-2r}m^2\right)$. 

Next we translate the $\chi^m$-divergence bound in \eqref{eq:chi_m_profile} into the probability bounds. Specifically, for any $S\subseteq \Phi_n\subseteq \cup_{t=1}^\infty \Phi_t$, the data processing inequality (or simply H\"{o}lder's inequality) gives
\begin{align}\label{eq:chi_m_data-processing}
\chi^m({\bf p}_h \| {\bf q}_h) \ge \frac{{\bf p}_h(S)^{m}}{{\bf q}_h(S)^{m-1}}. 
\end{align}
Furthermore, the profile probability under the Poissonized model and the original sampling model can be related as follows: for $S\subseteq \Phi_n$ and any non-negative discrete measure $q$, 
\begin{align}\label{eq:poisson-to-multi}
{\bf q}_h(S) = \nprob{\frac{q}{q(\NN)}}{S} \cdot \bP(\Poi(nq(\NN)) = n). 
\end{align}
Since $1\ge q(\NN)\ge (1-n^{-r})p(\NN) = 1-n^{-r}$ by \eqref{eq:lower_quantization}, we have
\begin{align*}
\bP(\Poi(nq(\NN)) = n) &\ge \bP(\Poi(n(1-n^{-r})) = n) = \bP(\Poi(n)=n)\cdot e^{n^{1-r}}(1-n^{-r})^n \\
&\ge \bP(\Poi(n)=n)\cdot (1-n^{-2r})^n \ge \exp(-c_rn^{1-2r}),
\end{align*}
for some constant $c_r>0$. Combining with the inequalities \eqref{eq:chi_m_profile}, \eqref{eq:chi_m_data-processing} and \eqref{eq:poisson-to-multi}, we conclude that
\begin{align*}
\nprob{p}{S} &\le \nprob{\frac{q}{q(\NN)}}{S}^{1-\frac{1}{m}}\exp\left(n^{1-2r}m + \frac{m-1}{m}c_rn^{1-2r}\right) \\
&= \nprob{\frac{q}{q(\NN)}}{S}^{1-\frac{1}{m}}\exp\left(O(n^{1-2r}m)\right). 
\end{align*}
Choosing $m=c_0n^s<n^r$ completes the proof of \eqref{eq:inequality_q_over_p} of Lemma \ref{lemma:continuity}. Interchanging the roles of $p$ and $q$ we may also arrive at the other inequality \eqref{eq:inequality_p_over_q}. 

\subsection{Proof of Lemma \ref{lemma:local_moments_expectation}}
Throughout the proof, we will use $C,C',\cdots$ to denote large positive constants depending only on $c_1$ which may vary from line to line, and $c,c',\cdots$ to denote the respective small positive constants. 

First we consider the case $d\ge 1$. Using the triangle inequality, we have
\begin{align}
\left|k\cdot \int_{\widetilde{I}_m} (x-x_m)^d \mu_m(dx) - \widehat{M}_{m,d} \right| &\le \left|k\cdot \int_{\widetilde{I}_m} (x-x_m)^d \mu_m(dx)  - M_{m,d} \right| \nonumber\\
&\qquad + \left|\bE[\widehat{M}_{m,d}] - M_{m,d} \right| + |\widehat{M}_{m,d} - \bE[\widehat{M}_{m,d}]|, \label{eq:decomposition}
\end{align}
where $M_{m,d}$ is the local smoothed moment defined in \eqref{eq:smoothed_moments}. We upper bound the expectation of each term in \eqref{eq:decomposition} separately. 

For the first term (which is deterministic), simple algebra gives
\begin{align*}
\left|k\cdot \int_{\widetilde{I}_m} (x-x_m)^d \mu_m(dx)  - M_{m,d} \right| \le \sum_{j=1}^k |p_j-x_m|^d \cdot  \bP\left(\mathsf{Poi}\left(\frac{np_j}{2}\right) \in \frac{n}{2}I_m\right)\1(p_j\notin \widetilde{I}_m). 
\end{align*}
By the Poisson tail inequality (cf. Lemma \ref{lemma:poissontail}), for $p_j\notin \widetilde{I}_m$ we have
\begin{align*}
\bP\left(\mathsf{Poi}\left(\frac{np_j}{2}\right) \in \frac{n}{2}I_m\right) \le \exp\left(-cn\cdot \min\left\{\frac{|x_m - p_j|^2}{p_j}, |x_m - p_j| \right\}\right)
\end{align*}
for some absolute constant $c>0$ independent of $(m,n,c_1,c_2)$. For $p_j\notin \widetilde{I}_m$ and constant $c_1>0$ large enough, simple algebra shows that for $1\le d\le D$, it holds that 
\begin{align*}
|p_j-x_m|^d \cdot \exp\left(-cn\cdot \min\left\{\frac{|x_m - p_j|^2}{p_j}, |x_m - p_j| \right\}\right) \le C\widetilde{\ell}_m^d\cdot n^{-5},
\end{align*}
for some constant $C>0$ depending only on $c_1$. Hence, the first term of \eqref{eq:decomposition} is upper bounded as
\begin{align}\label{eq:first_term}
\left|k\cdot \int_{\widetilde{I}_m} (x-x_m)^d \mu_m(dx)  - M_{m,d} \right| \le \frac{C}{n^5}\cdot \widetilde{\ell}_m^d. 
\end{align}

As for the second term, first note that
\begin{align*}
&\bE\left[\sum_{j=1}^k \sum_{s\in nI_m/2} \bP\left(\mathsf{B}\left(h_j,\frac{1}{2}\right)=s\right)\cdot g_{d,x_m}\left(\frac{h_j-s}{n/2}\right) \right] \\
&= \sum_{t=0}^\infty \sum_{j=1}^k \sum_{s\in nI_m/2} \frac{1}{2^t}\binom{t}{s} g_{d,x_m}\left(\frac{t-s}{n/2}\right)\cdot e^{-np_j}\frac{(np_j)^t}{t!} \\
&= \sum_{u=0}^\infty \sum_{j=1}^k \sum_{s\in nI_m/2} \frac{1}{2^{u+s}}\binom{u+s}{s} g_{d,x_m}\left(\frac{u}{n/2}\right)\cdot e^{-np_j}\frac{(np_j)^{u+s}}{(u+s)!} \\
&= \sum_{j=1}^k  \left(\sum_{s\in nI_m/2} e^{-np_j/2}\frac{(np_j/2)^s}{s!}\right) \left( \sum_{u=0}^\infty e^{-np_j/2}\frac{(np_j/2)^u}{u!}g_{d,x_m}\left(\frac{u}{n/2}\right)\right) \\
&= \sum_{j=1}^k \bP\left(\mathsf{Poi}\left(\frac{np_j}{2}\right)\in \frac{n}{2}I_m\right)\cdot (p_j - x_m)^d = M_{m,d},
\end{align*}
where we have used that $\bE[g_{d,x_m}(X)] = (p-x_m)^d$ whenever $nX/2\sim \mathsf{Poi}(np/2)$. Now by definition of the modification $\widetilde{g}_{d,x_m}$ in \eqref{eq:tilde_g_function}, we have
\begin{align}\label{eq:second_term_initial}
&|\bE[\widehat{M}_{m,d}] - M_{m,d}| \nonumber\\
&\le \bE\left[ \sum_{j=1}^k \sum_{s\in nI_m/2} \bP\left(\mathsf{B}\left(h_j,\frac{1}{2}\right)=s\right)\cdot \1\left(h_j - s \notin \frac{n}{2}[x_{m,\text{L}}, x_{m,\text{R}}]\right)\right. \nonumber\\
&\left. \qquad\qquad \cdot\left( \left|g_{d,x_m}\left(\frac{h_j-s}{n/2}\right)\right| + \left| g_{d,x_m}(x_{m,\text{L}})\right| + \left| g_{d,x_m}(x_{m,\text{R}})\right| \right) \right] \nonumber \\
&= \bE\left[ \sum_{j=1}^k \1\left(\widehat{p}_{j,1}\in I_m, \widehat{p}_{j,2}\notin [x_{m,\text{L}}, x_{m,\text{R}}]\right)\left( \left|g_{d,x_m}\left(\widehat{p}_{j,2}\right)\right| + \left| g_{d,x_m}(x_{m,\text{L}})\right| + \left| g_{d,x_m}(x_{m,\text{R}})\right| \right) \right],
\end{align}
where $\widehat{p}_{j,1}, \widehat{p}_{j,2}$ are independent random variables with $n\widehat{p}_{j,1}/2, n\widehat{p}_{j,2}/2 \sim \mathsf{Poi}(np_j/2)$, and the last identity follows from the subsampling property of Poisson random variables. To upper bound \eqref{eq:second_term_initial}, note that the condition of Lemma \ref{lemma:charlier} is satisfied with $\Delta\asymp \max\{\widetilde{\ell}_m, |\widehat{p}_{j,2} - x_m|\}$, and therefore 
\begin{align*}
\left|g_{d,x_m}\left(\widehat{p}_{j,2}\right)\right| + \left| g_{d,x_m}(x_{m,\text{L}})\right| + \left| g_{d,x_m}(x_{m,\text{R}})\right| \le Cn^{C'\cdot c_2} |\widehat{p}_{j,2} - x_m|^d
\end{align*}
for some absolute constants $C,C'>0$ depending only on $c_1$ whenever $\widehat{p}_{j,2}\notin [x_{m,\text{L}}, x_{m,\text{R}}]$. Moreover, if $p_j \in \widetilde{I}_m$, then the Poisson tail bound (cf. Lemma \ref{lemma:poissontail}) gives that 
\begin{align*}
\bP(\widehat{p}_{j,1}\in I_m, \widehat{p}_{j,2} = q) \le \bP(\widehat{p}_{j,2} = q) \le \exp\left( -cn\cdot \min\left\{\frac{|x_m - q|^2}{\max\{x_m,q\}}, |x_m - q|\right\} \right)
\end{align*}
for any $q\notin [x_{m,\text{L}}, x_{m,\text{R}}]$. If $p_j\notin \widetilde{I}_m$, then 
\begin{align*}
\bP(\widehat{p}_{j,1}\in I_m, \widehat{p}_{j,2} = q) &= \bP(\widehat{p}_{j,1}\in I_m)\cdot \bP(\widehat{p}_{j,2} = q) \\
&\le \exp\left( -cn\left(\min\left\{\frac{|p_j - x_m|^2}{p_j}, |p_j - x_m|\right\}+ \min\left\{\frac{|p_j - q|^2}{p_j}, |p_j - q|\right\} \right) \right) \\
&\le \exp\left(-c'n \cdot\min\left\{\frac{|x_m - q|^2}{\max\{x_m,q\}}, |x_m - q|\right\}\right). 
\end{align*}
In other words, the above inequality with a small enough constant $c'$ holds regardless of the true probability mass $p_j$. Since for $c_1>0$ large enough, for all $q\notin [x_{m,\text{L}}, x_{m,\text{R}}]$ it holds that
\begin{align*}
|q-x_m|^d \cdot \exp\left(-c'n\cdot \min\left\{\frac{|x_m - q|^2}{\max\{x_m,q\}}, |x_m - q| \right\}\right) \le C\widetilde{\ell}_m^d\cdot n^{-5},
\end{align*}
expanding the expectation in \eqref{eq:second_term_initial} gives the following upper bound on the second term of \eqref{eq:decomposition}: 
\begin{align}\label{eq:second_term}
|\bE[\widehat{M}_{m,d}] - M_{m,d}| \le \frac{Ck}{n^{4 - C'c_2}}\cdot \widetilde{\ell}_m^d. 
\end{align}

For the final term $\bE|\widehat{M}_{m,d} - \bE[\widehat{M}_{m,d}]|$, we will apply the Bennett inequality (cf. Lemma \ref{lemma:bennett}) to the sum of independent random variables $Z_j$, where
\begin{align*}
Z_j \triangleq \sum_{s\in nI_m/2} \bP\left(\mathsf{B}\left(h_j,\frac{1}{2}\right)=s\right)\cdot \widetilde{g}_{d,x_m}\left(\frac{h_j-s}{n/2}\right). 
\end{align*}
By Lemma \ref{lemma:charlier} and the definition of $\widetilde{g}_{d,x_m}$, we have $\|\widetilde{g}_{d,x_m}\|_\infty \le Cn^{C'c_2}\widetilde{\ell}_m^d$. As a result, 
\begin{align*}
|Z_j| \le Cn^{C'c_2}\widetilde{\ell}_m^d, \qquad \bE[Z_j^2] \le (Cn^{C'c_2}\widetilde{\ell}_m^d)^2 \cdot \bP\left(\mathsf{Poi}\left(\frac{np_j}{2}\right)\in \frac{n}{2}I_m\right). 
\end{align*}
Hence, by Bennett inequality, with probability at least $1-n^{-5}$ we have
\begin{align*}
\left| \sum_{j=1}^k (Z_j - \bE[Z_j]) \right| = O\left(n^{C'c_2}\widetilde{\ell}_m^d\left(\sqrt{\sum_{j=1}^k \bP\left(\mathsf{Poi}\left(\frac{np_j}{2}\right) \in \frac{n}{2}I_m\right) \cdot \log n }  + \log n \right)\right). 
\end{align*}
Consequently, using $|\sum_{j=1}^k Z_j|\le k\|\widetilde{g}_{d,x_m}\|_\infty$ almost surely, we conclude that
\begin{align}\label{eq:third_term}
\bE|\widehat{M}_{m,d} - \bE[\widehat{M}_{m,d}]| \le Cn^{C''c_2}\widetilde{\ell}_m^d\left(\sqrt{\sum_{j=1}^k \bP\left(\mathsf{Poi}\left(\frac{np_j}{2}\right) \in \frac{n}{2}I_m\right)}  + \frac{k}{n^5} \right). 
\end{align}
Hence, a combination of \eqref{eq:decomposition}, \eqref{eq:first_term}, \eqref{eq:second_term} and \eqref{eq:third_term} gives the claimed result for $d\ge 1$.

Next we consider the case where $d=0$. By the triangle inequality again, we have
\begin{align}
&\left|\sum_{m'\ge m}\left(k\cdot \mu_{m'}(\widetilde{I}_{m'}) - \widehat{M}_{m',0} \right) \right| \le \sum_{j=1}^k \sum_{m'\ge m} \bP\left( \mathsf{Poi}\left(\frac{np_j}{2}\right) \in \frac{n}{2}I_{m'} \right)\cdot \1(p_j\notin \widetilde{I}_{m'}) \nonumber\\
&\qquad  + \left|\sum_{j=1}^k \left(\bP\left( \mathsf{Poi}\left(\frac{np_j}{2}\right) \in \frac{n}{2}\bigcup_{m'\ge m}I_{m'} \right) - \bP\left(\mathsf{B}\left(h_j, \frac{1}{2}\right)\in \frac{n}{2}\bigcup_{m'\ge m}I_{m'}  \right) \right) \right|. \label{eq:decomposition_d=0}
\end{align}
By Lemma \ref{lemma.localization}, the first term of \eqref{eq:decomposition_d=0} is at most $k/n^4$. Regarding the second term, the subsampling property of Poisson random variables shows that each summand has a zero mean. Moreover, the random variable 
\begin{align*}
Y_j = \bP\left(\mathsf{B}\left(h_j, \frac{1}{2}\right)\in \frac{n}{2}\bigcup_{m'\ge m}I_{m'}  \right)
\end{align*}
is upper bounded by $1$. As for the variance of $Y_j$, we distinguish into three cases: 
\begin{itemize}
	\item If $p_j$ is greater than the right endpoint of $\widetilde{I}_m$, Lemma \ref{lemma.localization} shows that
	\begin{align*}
	\bP\left( \mathsf{Poi}\left(\frac{np_j}{2}\right) \in \frac{n}{2}\bigcup_{m'\ge m}I_{m'} \right) \ge 1-n^{-5}. 
	\end{align*}
	Therefore, 
	\begin{align*}
	\var(Y_j) \le \bE[(1-Y_j)^2] \le \bE[(1-Y_j)] = 1 - \bP\left( \mathsf{Poi}\left(\frac{np_j}{2}\right) \in \frac{n}{2}\bigcup_{m'\ge m}I_{m'} \right) \le n^{-5}. 
	\end{align*}
	\item Similarly, if $p_j$ is smaller than the left endpoint of $\widetilde{I}_m$, Lemma \ref{lemma.localization} shows that
	\begin{align*}
	\var(Y_j) \le \bE[Y_j^2] \le \bE[Y_j] = \bP\left( \mathsf{Poi}\left(\frac{np_j}{2}\right) \in \frac{n}{2}\bigcup_{m'\ge m}I_{m'} \right)\le n^{-5}. 
	\end{align*}
	\item Finally, if $p_j\in \widetilde{I}_m$, then Lemma \ref{lemma.localization} gives
	\begin{align*}
	\bP\left( \mathsf{Poi}\left(\frac{np_j}{2}\right) \in \frac{n}{2}\bigcup_{m'\ge m+1}I_{m'} \right) \le n^{-5}. 
	\end{align*}
	Therefore, 
	\begin{align*}
	\var(Y_j) \le \bE[Y_j^2] \le \bE[Y_j] \le \bP\left(\mathsf{Poi}\left(\frac{np_j}{2}\right)  \in \frac{n}{2}I_m \right) + n^{-5}. 
	\end{align*}
\end{itemize}
Combining all cases, it always holds that
\begin{align*}
\var(Y_j) \le \bP\left(\mathsf{Poi}\left(\frac{np_j}{2}\right)  \in \frac{n}{2}I_m \right) + n^{-5}. 
\end{align*}
Hence, by the Bennett inequality (cf. Lemma \ref{lemma:bennett}), with probability at least $1-n^{-5}$, 
\begin{align*}
\left| \sum_{j=1}^k (Y_j - \bE[Y_j]) \right| = O\left(\sqrt{\sum_{j=1}^k\left(  \bP\left(\mathsf{Poi}\left(\frac{np_j}{2}\right)  \in \frac{n}{2}I_m \right) + n^{-5}\right)\cdot \log n} + \log n \right).
\end{align*}
Therefore, the second term of \eqref{eq:decomposition_d=0} has expectation at most
\begin{align*}
\bE\left| \sum_{j=1}^k (Y_j - \bE[Y_j]) \right| \le C\left(\sqrt{\log n\cdot \sum_{j=1}^k \bP\left(\mathsf{Poi}\left(\frac{np_j}{2}\right) \in \frac{n}{2}I_m\right)}   + \log n + \frac{k}{n^5}\right),
\end{align*}
as claimed. 

\subsection{Proof of Lemma \ref{lemma:total_variance}}
Let $k_m' = \sum_{j=1}^k \1(p_j\in \widetilde{I}_m)$, then Lemma \ref{lemma.localization} gives $
k_m \le k_m' + k/n^5
$. Let $\calM = \{m\in [M]: k_m'\neq 0\}\subseteq[M]$. Since each probability mass $p_j$ lies in at most two different $\widetilde{I}_m$'s, we conclude that
\begin{align}\label{eq:sum}
2 = 2\sum_{j=1}^k p_j \ge \sum_{m\in \calM} \sum_{j=1}^k p_j\cdot \1(p_j\in \widetilde{I}_m) \ge \sum_{m\in \calM} \frac{c_1\log n}{n}\left(m-\frac{5}{4}\right)_+^2 \ge \frac{c_1\log n}{6n}(|\calM|-2)^3,
\end{align}
and consequently
\begin{align*}
|\calM| \le \left(\frac{12n}{c_1\log n}\right)^{\frac{1}{3}} + 2.
\end{align*}

Hence, 
\begin{align*}
\sum_{m=1}^M \widetilde{\ell}_m \sqrt{k_m} &\le \sum_{m=1}^M \widetilde{\ell}_m \left(\sqrt{k_m'} + \sqrt{\frac{k}{n^5}}\right) \\
&\le \sum_{m=1}^M \frac{2c_1m\log n}{n}\left(\sqrt{k_m'} + \sqrt{\frac{k}{n^5}}\right) \\
&\le \frac{2c_1\log n}{n}\sqrt{|\calM|\cdot \sum_{m\in \calM} m^2} + 2\sqrt{\frac{k}{n^5}} \\
&\stepa{=} O\left(\frac{\log n}{n}\cdot \sqrt{\left(\frac{n}{\log n}\right)^{\frac{1}{3}}\cdot \frac{n}{\log n}} + \sqrt{\frac{k}{n^5}} \right) \\
&= O\left(\left(\frac{\log n}{n}\right)^{\frac{1}{3}} + \sqrt{\frac{k}{n^5}}\right), 
\end{align*}
where step (a) follows from an intermediate step of \eqref{eq:sum}. 

\subsection{Proof of Lemma \ref{lemma:bounded_diff}}
Throughout the proof, we will frequently use the fact (which follows from Lemma \ref{lemma:charlier}) that
\begin{align}\label{eq:fact}
\|\widetilde{g}_{d,x_m}\|_{\infty} \le (x_{m,\text{R}} - x_{m,\text{L}})^d, \qquad \forall 0\le d\le D, m\in [M]. 
\end{align}

We first consider the case where $d\ge 1$. It is clear that
\begin{align*}
|\Delta_{m,d}| &= \left| \sum_{s\in nI_m/2}\left[ \bP\left(\mathsf{B}\left(h_j, \frac{1}{2}\right) = s\right)\widetilde{g}_{d,x_m}\left(\frac{h_j-s}{n/2}\right) - \bP\left(\mathsf{B}\left(h_j+1, \frac{1}{2}\right) = s\right)\widetilde{g}_{d,x_m}\left(\frac{h_j+1-s}{n/2}\right) \right]\right| \\
&\le \underbrace{ \sum_{s\in nI_m/2}  \bP\left(\mathsf{B}\left(h_j, \frac{1}{2}\right) = s\right)\cdot \left|\widetilde{g}_{d,x_m}\left(\frac{h_j-s}{n/2}\right)  -  \widetilde{g}_{d,x_m}\left(\frac{h_j+1-s}{n/2}\right)\right|  }_{=: D_1} \\
&\qquad + \underbrace{\sum_{s\in nI_m/2} \left|  \bP\left(\mathsf{B}\left(h_j, \frac{1}{2}\right) \in \frac{n}{2}I_m\right) - \bP\left(\mathsf{B}\left(h_j+1, \frac{1}{2}\right) \in \frac{n}{2}I_m\right) \right| \cdot \left| \widetilde{g}_{d,x_m}\left(\frac{h_j+1-s}{n/2}\right)\right|}_{=: D_2}. 
\end{align*}
We distinguish into two cases $m\in \calM$ and $m\notin \calM$, respectively. 
\begin{itemize}
	\item Case I: $m \in \calM$. By the definition of $\widetilde{g}_{d,x_m}$ and Lemma \ref{lemma:charlier}, we have
	\begin{align*}
	\left|\widetilde{g}_{d,x_m}\left(\frac{h_j-s}{n/2}\right)  -  \widetilde{g}_{d,x_m}\left(\frac{h_j+1-s}{n/2}\right)\right| &\le \frac{2d}{n}\left| \widetilde{g}_{d-1,x_m}\left(\frac{h_j-s}{n/2}\right) \right| \le \frac{2d}{n}(x_{m,\text{R}} - x_{m,\text{L}})^{d-1},
	\end{align*}
	where the last inequality follows from \eqref{eq:fact}. Hence, 
	\begin{align*}
	D_1 \le \frac{2d}{n}(x_{m,\text{R}} - x_{m,\text{L}})^{d-1} \le \frac{2D}{n}(x_{m,\text{R}} - x_{m,\text{L}})^{d-1}. 
	\end{align*}
	As for $D_2$, Lemma \ref{lemma:cdf_diff} gives that
	\begin{align*}
	\left| \bP\left(\mathsf{B}\left(h_j, \frac{1}{2}\right) \in \frac{n}{2}I_m\right) - \bP\left(\mathsf{B}\left(h_j+1, \frac{1}{2}\right) \in \frac{n}{2}I_m\right) \right| \le \frac{2C}{\sqrt{h_j}}. 
	\end{align*}
	Since $m\in \calM$, for $m\ge 2$ we have $h_j\ge nx_{m,\text{L}}\ge n^2(x_{m,\text{R}} - x_{m,\text{L}})^2/(144c_1\log n)$. Moreover, for $m=1$ the above probability difference is trivially upper bounded by $1$. Combining these two cases, we always have
	\begin{align*}
	\left| \bP\left(\mathsf{B}\left(h_j, \frac{1}{2}\right) \in \frac{n}{2}I_m\right) - \bP\left(\mathsf{B}\left(h_j+1, \frac{1}{2}\right) \in \frac{n}{2}I_m\right) \right| = O\left( \frac{\log n}{n(x_{m,\text{R}} - x_{m,\text{L}} )} \right).
	\end{align*}
	With the help of \eqref{eq:fact} again, we conclude that
	\begin{align*}
	|\Delta_{m,d}| \le  D_1 + D_2 = O\left(\frac{\log n}{n}\cdot (x_{m,\text{R}} - x_{m,\text{L}})^{d-1}\right). 
	\end{align*}
	\item Case II: $m\notin \calM$. By \eqref{eq:fact}, it is clear that
	\begin{align*}
	|\Delta_{m,d}| \le (x_{m,\text{R}} - x_{m,\text{L}})^{d}\cdot \left(\bP\left(\mathsf{B}\left(h_j, \frac{1}{2}\right) \in \frac{n}{2}I_m\right) + \bP\left(\mathsf{B}\left(h_j+1, \frac{1}{2}\right) \in \frac{n}{2}I_m\right) \right).
	\end{align*}
	Thanks to the assumption $h_j\notin n[x_{m,\text{L}}, x_{m,\text{R}}]$, the Binomial tail inequality (cf. Lemma \ref{lemma:poissontail}) yields that for $c_1>0$ large enough, we have
	\begin{align*}
	\bP\left(\mathsf{B}\left(h_j, \frac{1}{2}\right) \in \frac{n}{2}I_m\right) + \bP\left(\mathsf{B}\left(h_j+1, \frac{1}{2}\right) \in \frac{n}{2}I_m\right) \le \frac{2}{n^5}. 
	\end{align*}
	Hence, for $m\notin \calM$ we have
	\begin{align*}
	|\Delta_{m,d}|  = O\left( \frac{(x_{m,\text{R}} - x_{m,\text{L}})^{d}}{n^5} \right). 
	\end{align*}
\end{itemize}

Next we deal with the case $d=0$, and it is clear that
\begin{align*}
\left|\sum_{m'\ge m} \Delta_{m,0} \right| &= \left|\bP\left(\mathsf{B}\left(h_j, \frac{1}{2} \right) \in \frac{n}{2}\bigcup_{m'\ge m}I_{m'}\right) - \bP\left(\mathsf{B}\left(h_j+1, \frac{1}{2} \right) \in \frac{n}{2}\bigcup_{m'\ge m}I_{m'} \right)\right| \\
&= \left|\bP\left(\mathsf{B}\left(h_j, \frac{1}{2} \right) \ge \frac{c_1(m-1)^2\log n}{2} \right) - \bP\left(\mathsf{B}\left(h_j+1, \frac{1}{2} \right) \ge \frac{c_1(m-1)^2\log n}{2} \right)\right|. 
\end{align*}
If $m\in \calM$, we use Lemma \ref{lemma:cdf_diff} again to conclude that
\begin{align*}
\left|\bP\left(\mathsf{B}\left(h_j, \frac{1}{2} \right) \ge \frac{c_1(m-1)^2\log n}{2} \right) - \bP\left(\mathsf{B}\left(h_j+1, \frac{1}{2} \right) \ge \frac{c_1(m-1)^2\log n}{2} \right)\right| \le \frac{C}{\sqrt{h_j}},
\end{align*}
which by the same analysis above yields to an upper bound of $O(\log n/(n (x_{m,\text{R}} - x_{m,\text{L}})))$. For the other case $m\notin \calM$, the Binomial tail inequality (cf. Lemma \ref{lemma:poissontail}) again shows that for $c_1>0$ large enough, in the respective scenarios $h_j > nx_{m,\text{R}}$ and $h_j < nx_{m,\text{L}}$, the above Binomial probabilities are both $\ge 1-n^{-5}$ or $\le n^{-5}$, respectively. Hence, for $m\notin \calM$ we have
\begin{align*}
\left|\sum_{m'\ge m} \Delta_{m,0} \right| = O(n^{-5}), 
\end{align*}
as desired. 

\subsection{Proof of Theorem \ref{thm.poisson_approx}}
We shall construct the following local intervals: for $m=1,2,\cdots,M = \sqrt{n/(c_1\log n)}$, define
\begin{align*}
I_m &= \left[\frac{c_1\log n}{n}\cdot (m-1)^2, \frac{c_1\log n}{n}\cdot m^2 \right], \\
I_m' &= \left[\frac{c_1\log n}{n}\cdot \left(m-\frac{4}{3}\right)_+^2, \frac{c_1\log n}{n}\cdot \left(m+\frac{1}{3}\right)^2 \right], \\
I_m'' &= \left[\frac{c_1\log n}{n}\cdot \left(m-2\right)_+^2, \frac{c_1\log n}{n}\cdot \left(m+1\right)^2 \right],
\end{align*}
where $c_1>0$ is some large constant to be specified later, and without loss of generality we assume that $M$ is an integer. We shall also define
$$
x_m = \frac{c_1\log n}{n}\cdot \left( m - \frac{1}{2}\right)^2
$$
to be the center of the $m$-th local interval. Note that $I_m\subseteq I_m' \subseteq I_m''$, and by Lemma \ref{lemma.localization} we have
\begin{align}\label{eq:local_intervals}
\begin{split}
\bP(\Poi(nx)\notin nI_m' \mid x\in I_m ) &\le n^{-5}, \\
\bP(\Poi(nx)\notin nI_m'' \mid x\in I_m') &\le n^{-5}, \\
\bP(\Poi(nx)\notin nI_m \mid x\in I_m - I_{m-1}' - I_{m+1}') &\le n^{-5}. 
\end{split}
\end{align}

\subsubsection{Construction of Local Poisson Polynomials}
We first construct a local Poisson polynomial on each local interval $I_m'$. 
\begin{lemma}\label{lemma.local_approximation}
	Let $f$ be any $1$-Lipschitz function, $\varepsilon>0$, and $c_1>0$ be large enough. Then there exists a constant $C>0$ such that for each $m\in [M]$, there is a sequence of coefficients $(b_j)_{j=0}^{\infty}$ with
	\begin{align}\label{eq.approx_error_local}
	\left|f(x) - \sum_{j=0}^{\infty} b_j\bP(\Poi(nx) = j) \right| \le C\sqrt{\frac{x}{n\log n}}, \quad \forall x\in I_m',
	\end{align}
	where $b_j = 0$ for $j\notin nI_m''$, and
	\begin{align}\label{eq.coeff_bound_local}
	\left|b_j - f\left(\frac{j}{n}\right) \right| \le \frac{C(1+j^{1/2})}{n^{1-\varepsilon}}, \quad \forall j\in nI_m''.
	\end{align}
\end{lemma}

It is clear that Lemma \ref{lemma.local_approximation} is an analog of Theorem \ref{thm.poisson_approx} on local intervals $I_m'$, where we additionally require that $b_j = 0$ for all $j\notin nI_m''$ to essentially shut down the influence on other local intervals, which will be crucial for the next step. The high-level idea of the proof of Lemma \ref{lemma.local_approximation} is as follows: we apply the best polynomial approximation with degree $O(\log n)$ on the local interval $I_m'$, and then convert the monomial basis into the Poisson polynomial basis. To fulfill the support condition of $(b_j)$ as well as the inequality \eqref{eq.coeff_bound_local}, we shall need a careful truncation argument to only keep the coefficients in the local interval $I_m''$. 

\begin{proof}[Proof of Lemma \ref{lemma.local_approximation}]
	Let $D=\lceil c_2\log n \rceil$ where $c_2>0$ is a small constant depending only on $\varepsilon$. Throughout the proof we will use $C_1,C_2,\cdots$ to denote positive numerical constants independent of $(n,\varepsilon,c_1)$. 
	
	By Lemma \ref{lemma:jackson}, there exists a degree-$D$ polynomial such that
	$$
	\left| f(x) -\sum_{d=0}^D a_d(x-x_m)^d \right| \le C_1\sqrt{\frac{x}{n\log n}} 
	$$
	holds for all $x\in I_m'$, with
	$$
	|a_d| \le \left(C_2\cdot \frac{c_1m\log n}{n} \right)^{1-d}, \quad d=1,2,\cdots,D.
	$$
	As for $d=0$, choosing $x=x_m$ in the above inequality gives $a_0 = f(x_m) + O(m/n)$. 
	
	Next we write the above polynomial as a linear combination of Poisson polynomials. Since \cite[Example 2.8]{Withers1987}
	$$
	\sum_{j=0}^\infty \frac{j!}{(j-d)!n^d}\cdot \bP(\Poi(nx)=j) = x^d,
	$$
	we have
	\begin{align*}
	\sum_{d=0}^D a_d(x-x_m)^d &= \sum_{d=0}^D a_k \sum_{d'=0}^d \binom{d}{d'}(-x_m)^{d-d'}x^{d'} \\
	&=  \sum_{d=0}^D a_d \sum_{d'=0}^d \binom{d}{d'}(-x_m)^{d-d'} \sum_{j=0}^\infty \frac{j!}{(j-d')!n^{d'}}\cdot \bP(\Poi(nx)=j) \\
	&= \sum_{j=0}^\infty \underbrace{\left( \sum_{d=0}^D a_d\sum_{d'=0}^d \binom{d}{d'}(-x_m)^{d-d'} \frac{j!}{(j-d')!n^{d'}} \right)}_{\triangleq b_j^\star} \bP(\Poi(nx)=j). 
	\end{align*}
	In other words, the inequality \eqref{eq.approx_error_local} holds for the coefficients $(b_j^\star)_{j=0}^\infty$. Now we define $(b_j)_{j=0}^\infty$ to be the truncated version of $(b_j^\star)_{j=0}^\infty$: 
	$$
	b_j = b_j^\star\cdot \1(j\in nI_m''). 
	$$
	Clearly $b_j = 0$ for all $j\notin nI_m''$. By Lemma \ref{lemma:charlier}, for $d=1,2,\cdots,D$, 
	\begin{align*}
	\left|\sum_{d'=0}^{d} \binom{d}{d'}(-x_m)^{d-d'}\frac{j!}{(j-d')!n^{d'}} \right| \le \begin{cases}
	\left(C_3\cdot \frac{c_1m\log n}{n}\right)^d &\text{if } j\in nI_m'', \\
	\left(C_4\cdot |j/n - x_m|\right)^d & \text{otherwise}. 
	\end{cases}
	\end{align*}
	Hence, for $j\in nI_m''$, we have
	\begin{align*}
	|b_j - a_0| = |b_j^\star - a_0| &\le \sum_{d=1}^D |a_d|\cdot \left|\sum_{d'=0}^{d} \binom{d}{d'}(-x_m)^{d-d'}\frac{j!}{(j-d')!n^{d'}} \right| \\
	&\le \sum_{d=1}^D \left(C_2\cdot \frac{c_1m\log n}{n}\right)^{1-d}\cdot \left(C_3\cdot \frac{c_1m\log n}{n}\right)^d \\
	&\le DC_2\left(1 + (C_3/C_2)^D\right)\cdot \frac{c_1 m\log n}{n} \\
	&=O\left( \frac{1+ j^{1/2}}{n^{1-\varepsilon}} \right)
	\end{align*}
	where the last step follows from the choice of $D=c_2\log n$ with a sufficiently small $c_2>0$ and the assumption $j\in nI_m''$. Moreover, for any $j\in nI_m''$,
	\begin{align*}
	a_0 = f(x_m) + O\left(\frac{m}{n}\right) &= f\left(\frac{j}{n}\right) + \left|x_m - \frac{j}{n}\right| + O\left(\frac{m}{n}\right) \\
	&= f\left(\frac{j}{n}\right) + O\left(\frac{m\log n}{n}\right) = f\left(\frac{j}{n}\right) + O\left(\frac{(1+j^{1/2})\log n}{n}\right),
	\end{align*}
	and therefore a triangle inequality gives the inequality \eqref{eq.coeff_bound_local}.
	
	As for the other inequality \eqref{eq.approx_error_local}, by triangle inequality it suffices to prove that
	\begin{align}\label{eq.truncation_error}
	\sum_{j\notin nI_m''} |b_j^\star|\cdot \bP(\Poi(nx)=j)  = O(n^{-4}), \quad \forall x\in I_m'. 
	\end{align}
	To prove \eqref{eq.truncation_error}, first note that for $j\notin nI_m''$, we have
	\begin{align*}
	|b_j^\star| &\le |a_0| + \sum_{d=1}^D \left(C_2\cdot \frac{c_1m\log n}{n}\right)^{1-d}\left(C_4\cdot \left|\frac{j}{n} - x_m\right|\right)^d \\
	&= O\left(1+ \left(C_5\cdot \frac{|j-nx_m|}{c_1m\log n}\right)^D \right) = O\left(1+ \left(C_6\cdot \frac{|j-nx_m|}{\sqrt{c_1nx_m\log n}}\right)^D \right). 
	\end{align*}
	Furthermore, by the Chernoff bound (cf. Lemma \ref{lemma:poissontail}), for all $x\in I_m'$ and $j\notin nI_m''$ we have
	\begin{align*}
	\bP(\Poi(nx) = j) &\le \exp\left( - C_7\left(\frac{(j-nx)^2}{nx} \wedge |j-nx|  \right)\right) \\
	&\le \exp\left( - C_8\cdot c_1\log n \cdot \frac{|j-nx_m|}{\sqrt{c_1nx_m\log n}} \right). 
	\end{align*}
	Moreover, the assumption $j\notin nI_m''$ implies that $|j-nx_m| / \sqrt{c_1nx_m\log n} \ge C_9>0$. Consequently, whenever $p\in I_m'$ and $j\notin nI_m''$, we have
	$$
	|b_j^\star|\cdot \bP(\Poi(np) = j) = O\left(n^{-5} + \exp\left( c_2\log n\cdot \log\left(C_6\cdot \frac{|j-nx_m|}{\sqrt{c_1nx_m\log n}}\right) -  C_8c_1\log n\cdot \frac{|j-nx_m|}{\sqrt{c_1nx_m\log n}}  \right) \right), 
	$$
	where the first $O(n^{-5})$ term follows from \eqref{eq:local_intervals}. Since the positive constants $C_6$ and $C_8$ do not depend on the parameter $c_1$, by choosing $c_1>0$ large enough we arrive at an exponent $\le -5\log n$, and therefore $|b_j^\star|\cdot \bP(\Poi(np) = j) = O(n^{-5})$. Now the inequality \eqref{eq.truncation_error} simply follows from the summation over $O(n)$ indices. 
\end{proof}

\subsubsection{Piecing together Local Polynomials}
Motivated by \eqref{eq:functional_no_sample_split}, we assume that $\sum_{j=0}^\infty b_j^{(m)}\bP(\Poi(nx/2)=j)$ is the Poisson polynomial given by Lemma \ref{lemma.local_approximation} on the $m$-th local interval $I_m'$, with $n$ replaced by $n/2$. Now consider the following Poisson polynomial: 
$$
F(x) \triangleq \sum_{j=0}^\infty b_j\bP(\Poi(nx) = j),
$$ 
with 
\begin{align}\label{eq.b_j}
b_j \triangleq \frac{1}{2^j}\sum_{m=1}^M \sum_{k\in nI_m/2}\binom{j}{k}b_{j-k}^{(m)}. 
\end{align}
We claim that the above polynomial with coefficients in \eqref{eq.b_j} satisfies Theorem \ref{thm.poisson_approx}. We first verify the inequality \eqref{eq.approx_error}. Using a change of variable $l = j-k$, we have
\begin{align*}
F(x) &= \sum_{j=0}^\infty \frac{1}{2^j}\sum_{m=1}^M\sum_{k\in nI_m/2} \binom{j}{k}b_{j-k}^{(m)}\cdot \bP(\Poi(nx)=j) \\
&= \sum_{m=1}^M \sum_{l=0}^\infty b_l^{(m)} \sum_{k\in nI_m/2} \frac{1}{2^{k+l}}\binom{k+l}{k}\cdot \bP(\Poi(nx)=k+l) \\
&= \sum_{m=1}^M \sum_{l=0}^\infty b_l^{(m)} \sum_{k\in nI_m/2} \frac{1}{2^{k+l}}\cdot e^{-nx}\frac{(nx)^{k+l}}{k!l!}\\
&= \sum_{m=1}^M \sum_{l=0}^\infty b_l^{(m)}\bP(\Poi(nx/2) = l) \sum_{k\in nI_m/2} \frac{1}{2^{k}}\cdot e^{-nx/2}\frac{(nx)^{k}}{k!} \\
&= \sum_{m=1}^M \bP(\Poi(nx/2) \in nI_m/2)\sum_{l=0}^\infty b_l^{(m)}\bP(\Poi(nx/2) = l) .
\end{align*}
Since $I_m$ constitutes a partition of $[0,1]$, for $x\in [0,1]$ there exists $m^\star \in [M]$ such that $p\in I_{m^\star}$. We distinguish into three cases: 
\begin{enumerate}
	\item Case I: $x\in I_{m^\star} - I_{m^\star-1}' - I_{m^\star+1}'$. By \eqref{eq:local_intervals}, we have $\bP(\Poi(nx/2)\notin nI_{m^\star}/2) \le n^{-5}$, and therefore $\bP(\Poi(nx/2)\in nI_m/2) \le n^{-5}$ for any $m\neq m^\star$. Hence, 
	\begin{align*}
	|f(x) - F(x)| &\le \left|f(x) - \sum_{l=0}^\infty b_l^{(m^\star)}\bP(\Poi(nx/2) = l)\right| \\
	&\qquad + \bP(\Poi(nx/2)\notin nI_{m^\star}/2)\cdot \left|\sum_{l=0}^\infty b_l^{(m^\star)}\bP(\Poi(nx/2) = l)\right| \\
	&\qquad + \sum_{m\neq m^\star} \bP(\Poi(nx/2)\in nI_m/2)  \left|\sum_{l=0}^\infty b_l^{(m)}\bP(\Poi(nx/2) = l)\right| \\
	&\le C\sqrt{\frac{x}{n\log n}} + n^{-5}\sum_{m=1}^M \frac{C}{n^{1/2-\varepsilon}},
	\end{align*}
	where we have used \eqref{eq.coeff_bound_local} in the last inequality. As a result, the desired approximation error in \eqref{eq.approx_error} holds if $x\ge 1/n$. For $x<1/n$, the failure probability may be strengthened from $O(n^{-5})$ to $O(n^{-4}x)$, still giving the desired bound.
	\item Case II: $x\in I_{m^\star} \cap I_{m^\star + 1}'$. In this case, \eqref{eq:local_intervals} gives $\bP(\Poi(nx/2)\in nI_m/2) \le n^{-5}$ for any $m\notin \{m^\star, m^\star+1\}$. Consequently, 
	\begin{align*}
	|f(x) - F(x)| &\le \bP(\Poi(nx/2)\in nI_{m^\star}/2)\left|f(x) - \sum_{l=0}^\infty b_l^{(m^\star)}\bP(\Poi(nx/2) = l)\right| \\
	&\qquad + \bP(\Poi(nx/2)\in nI_{m^\star+1}/2)\left|f(x) - \sum_{l=0}^\infty b_l^{(m^\star+1)}\bP(\Poi(nx/2) = l)\right|  \\
	&\qquad + \sum_{m\neq m^\star, m^\star+1} \bP(\Poi(nx/2)\in nI_m/2)  \left|\sum_{l=0}^\infty b_l^{(m)}\bP(\Poi(nx/2) = l)\right|,
	\end{align*}
	and using Lemma \ref{lemma.local_approximation} and the same concentration bounds gives \eqref{eq.approx_error}. 
	\item Case III: $x\in I_{m^\star} \cap I_{m^\star - 1}'$. This case is entirely symmetric to Case II. 
\end{enumerate}
Combining the above three cases, we arrive at the inequality \eqref{eq.approx_error}. 

Next we verify the coefficient bound \eqref{eq.coeff_bound}. By Lemma \ref{lemma.local_approximation}, it is clear from the definition that $b_j=0$ whenever $j\notin \cup_{m=1}^M nI_m''$. Fix any $j\ge 0$ such that $b_j\neq 0$, assume that $j\in nI_{m^\star}''$ (if there are multiple choices of $m^\star$, pick an arbitrary one). We claim that any other $m\in [M]$ such that $|m-m^\star|\ge 5$ do not contribute to $b_j$ in the summation \eqref{eq.b_j}. In fact, if there is a non-zero coefficient $b_{j-k}^{(m)}$ in \eqref{eq.b_j}, we must have
\begin{align*}
j \in nI_{m^\star}'' &= c_1\log n\cdot \left[\left(m^\star - 2\right)_+^2, (m^\star+1)^2 \right], \\
k \in nI_{m}/2 &= \frac{c_1\log n}{2}\cdot \left[\left(m-1\right)^2, m^2 \right], \\
j-k \in nI_{m}''/2 &= \frac{c_1\log n}{2}\cdot  \left[\left(m - 2\right)_+^2, (m+1)^2 \right]. 
\end{align*}
Summing up, we must have
\begin{align*}
(m^\star-2)_+^2 &\le \frac{(m-1)^2 + (m-2)_+^2}{2}, \\
(m^\star+1)^2 &\ge \frac{m^2 + (m+1)^2}{2},
\end{align*}
at least one of which will fail whenever $|m-m^\star|\ge 5$. Hence, 
\begin{align*}
|b_j| \le  \frac{1}{2^j}\sum_{m=1}^M \sum_{k\in nI_m/2}\binom{j}{k}|b_{j-k}^{(m)}| \le \max_{m: |m-m^\star|\le 4} \max_{l\ge 0} |b_l^{(m)}| \le \frac{C_0(1+m^\star)}{n^{1-\varepsilon}} \le \frac{C(1+j^{1/2})}{n^{1-\varepsilon}},
\end{align*}
establishing \eqref{eq.coeff_bound}.

\end{document}